\documentclass[secthm,seceqn,amsthm,ussrhead,reqno, 12pt]{amsart}
\usepackage[utf8]{inputenc}
\usepackage[english]{babel}
\usepackage[symbol]{footmisc}
\usepackage{amssymb,amsmath,amsthm,amsfonts,xcolor,enumerate,hyperref,comment,longtable,cleveref}

\usepackage{times}
\usepackage{cite}
\usepackage{pdflscape}
\usepackage{ulem}
\usepackage[mathcal]{euscript}
\usepackage{tikz}
\usepackage{pgfplots}
\pgfplotsset{compat=1.15}
\usepackage{mathrsfs}
\usetikzlibrary{arrows}
\usepackage{cancel}
\usepackage{stmaryrd}

\newtheorem{thrm}{Theorem}[section] 
\newtheorem{cor}[thrm]{Corollary}
\newtheorem{lem}[thrm]{Lemma}
\newtheorem{prop}[thrm]{Proposition}

\theoremstyle{definition}
\newtheorem{defn}[thrm]{Definition}
\newtheorem{exm}[thrm]{Example}

\newtheorem{rem}[thrm]{Remark}

\crefrangeformat{equation}{#3(#1)#4--#5(#2)#6}

\crefname{thrm}{Theorem}{Theorems}
\crefname{lem}{Lemma}{Lemmas}
\crefname{cor}{Corollary}{Corollaries}
\crefname{prop}{Proposition}{Propositions}
\crefname{defn}{Definition}{Definitions}
\crefname{exm}{Example}{Examples}
\crefname{rem}{Remark}{Remarks}
\crefname{section}{Section}{Sections}
\crefname{equation}{\unskip}{\unskip}
\crefname{enumi}{\unskip}{\unskip}

\DeclareMathOperator{\Ann}{Ann}
\DeclareMathOperator{\Aut}{Aut}
\DeclareMathOperator{\Max}{Max}
\DeclareMathOperator{\Min}{Min}

\DeclareMathOperator{\sgn}{sgn}

\renewcommand{\iff}{\Leftrightarrow}
\newcommand{\impl}{\Rightarrow}
\newcommand{\gen}[1]{\langle #1\rangle}

\newcommand{\dl}{\delta}

\newcommand{\Dl}{\Delta}
\newcommand{\G}{\Gamma}
\newcommand{\lb}{\lambda}
\newcommand{\sg}{\sigma}
\newcommand{\vf}{\varphi}
\newcommand{\kp}{\kappa}

\newcommand{\LL}{\mathfrak{L}}

\newcommand{\sst}{\subseteq}

\newcommand{\tl}{\tilde}

\newcommand{\ad}{\mathrm{ad}}
\newcommand{\ch}{\mathrm{char}}

\newcommand{\m}{{}^{-1}}

\begin{document}

	\noindent{\Large  
		Transposed Poisson structures on Lie incidence algebras}\footnote{
		The work was supported by  
		FCT   UIDB/MAT/00212/2020, UIDP/MAT/00212/2020, 2022.02474.PTDC and by CMUP, member of LASI, which is financed by national funds through FCT --- Fundação para a Ciência e a Tecnologia, I.P., under the project with reference UIDB/00144/2020.} 
	\footnote{Corresponding author: kaygorodov.ivan@gmail.com}
	
	\bigskip

	\begin{center}
	    
	{\bf
		Ivan Kaygorodov\footnote{CMA-UBI, Universidade da Beira Interior, Covilh\~{a}, Portugal; \    kaygorodov.ivan@gmail.com}	\&
		Mykola Khrypchenko\footnote{ Departamento de Matem\'atica, Universidade Federal de Santa Catarina,     Brazil; \and CMUP, Departamento de Matemática, Faculdade de Ciências, Universidade do Porto,
			Rua do Campo Alegre s/n, 4169--007 Porto, Portugal\ nskhripchenko@gmail.com}
	}
	\

	\end{center}

	\bigskip
	
	\ 
	
	\noindent {\bf Abstract:} {\it 	
		Let $X$ be a finite connected poset, $K$ a field of characteristic zero and $I(X,K)$ the incidence algebra of $X$ over $K$ seen as a Lie algebra under the commutator product. In the first part of the paper we show that any $\frac 12$-derivation of $I(X,K)$ decomposes into the sum of a central-valued $\frac 12$-derivation, an inner $\frac 12$-derivation and a $\frac 12$-derivation associated with a map $\sg:X^2_<\to K$ that is constant on chains and cycles in $X$. In the second part of the paper we use this result to prove that any transposed Poisson structure on $I(X,K)$ is the sum of a structure of Poisson type, a mutational structure and a structure determined by $\lb:X^2_e\to K$, where $X^2_e$ is the set of $(x,y)\in X^2$ such that $x<y$ is a maximal chain not contained in a cycle.

	}

	\
	

	
	\tableofcontents

	%\newpage
	\section*{Introduction}
	
	Since their origin in the 1970s in Poisson geometry, Poisson algebras have appeared in several areas of mathematics and physics, such as algebraic geometry, operads, quantization theory, quantum groups, and classical and quantum mechanics. One of the natural tasks in the theory of Poisson algebras is the description of all such algebras with fixed Lie or associative part~\cite{YYZ07,jawo,kk21}.

	Recently, Bai, Bai, Guo, and Wu~\cite{bai20} have introduced a dual notion of the Poisson algebra, called a \textit{transposed Poisson algebra}, by exchanging the roles of the two multiplications in the Leibniz rule defining a Poisson algebra. A transposed Poisson algebra defined this way not only shares some properties of a Poisson algebra, such as the closedness under tensor products and the Koszul self-duality as an operad, but also admits a rich class of identities \cite{kms,bai20,fer23,bfk22}. It is important to note that a transposed Poisson algebra naturally arises from a Novikov-Poisson algebra by taking the commutator Lie algebra of its Novikov part \cite{bai20}.
 	Any unital transposed Poisson algebra is
	a particular case of a ``contact bracket'' algebra 
	and a quasi-Poisson algebra \cite{bfk22}.
 Each transposed Poisson algebra is a 
 commutative  Gelfand-Dorfman algebra \cite{kms}
 and it is also an algebra of Jordan brackets \cite{fer23}.

	In a recent paper by Ferreira, Kaygorodov, and  Lopatkin
	a relation between $\frac{1}{2}$-derivations of Lie algebras and 
	transposed Poisson algebras has been established \cite{FKL}. 	These ideas were used to describe all transposed Poisson structures 
	on  Witt and Virasoro algebras in  \cite{FKL};
	on   twisted Heisenberg-Virasoro,   Schr\"odinger-Virasoro  and  
	extended Schr\"odinger-Virasoro algebras in \cite{yh21};
	on Schr\"odinger algebra in $(n+1)$-dimensional space-time in \cite{ytk};
	on Witt type Lie algebras in \cite{kk23};
	on generalized Witt algebras in \cite{kkg23}; 
 Block Lie algebras in \cite{kk22,kkg23}
 and
 on the Lie algebra of upper triangular matrices in \cite{KK7}.
		Any complex finite-dimensional solvable Lie algebra was proved to admit a non-trivial transposed Poisson structure \cite{klv22}.
	The algebraic and geometric classification of $3$-dimensional transposed Poisson algebras was given in \cite{bfk23}.	
	For the list of actual open questions on transposed Poisson algebras, see \cite{bfk22}.	

	In this paper, we give a full characterization of transposed Poisson structures on the incidence Lie algebra $(I(X,K),[\cdot,\cdot])$ of a finite connected poset $X$ over a field $K$ of characteristic zero. As usual, our result is based on the description of $\frac 12$-derivations of $(I(X,K),[\cdot,\cdot])$, which is itself interesting. Apart from the \textit{central-valued} and \textit{inner} $\frac 12$-derivations that can be defined on an arbitrary Lie algebra (although, they are not always non-trivial), there is also a class of $\frac 12$-derivations depending on the combinatorial structure of $X$. Such $\frac 12$-derivations are determined by maps $\sg:X^2_<\to K$ that are constant on chains and cycles in $X$. It is proved in \cref{descr-half-der-I(X_K)} that any $\frac 12$-derivation of $(I(X,K),[\cdot,\cdot])$ is uniquely represented as the sum of $\frac 12$-derivations of the $3$ above mentioned types. This is the main result of the first part of the paper.
	
	Each of the $3$ classes of $\frac 12$-derivations of $I(X,K)$ corresponds to a class of transposed Poisson structures on $I(X,K)$. Central-valued $\frac 12$-derivations correspond to the structures of \textit{Poisson type}. Inner $\frac 12$-derivations correspond to the \textit{mutational} structures. And the $\frac 12$-derivations associated with $\sg:X^2_<\to K$ give rise to the so-called \textit{$\lb$-structures} determined by maps $\lb:X^2_e\to K$, where $X^2_e$ is the set of $(x,y)\in X^2$ such that $x<y$ is a maximal chain not contained in a cycle. The sum of such structures is again a transposed Poisson structure on $I(X,K)$, and any transposed Poisson structure on $I(X,K)$ can be decomposed into the sum. This is proved in \cref{descr-TP-on-T(X_K)}, which is the main result of the second part of the paper. We finish our article with a series of examples calculating transposed Poisson structures on the incidence algebras of concrete finite connected posets.

	%	%\newpage	
	\section{Definitions and preliminaries}\label{prelim}
	
	\subsection{Transposed Poisson algebras}
	All the algebras below will be over a field $K$ of characteristic zero and all the linear maps will be $K$-linear, unless otherwise stated. The notation $\gen{S}$ means the $K$-subspace generated by $S$.

	\begin{defn}[Bai, Bai, Guo, and Wu \cite{bai20}]\label{tpa}
		Let ${\mathfrak L}$ be a vector space equipped with two bilinear operations $\cdot$ and $[\cdot,\cdot].$
		The triple $({\mathfrak L},\cdot,[\cdot,\cdot])$ is called a \textit{transposed Poisson algebra} if $({\mathfrak L},\cdot)$ is a commutative associative algebra and
		$({\mathfrak L},[\cdot,\cdot])$ is a Lie algebra satisfying
		\begin{align}\label{Trans-Leibniz}
			2z\cdot [x,y]=[z\cdot x,y]+[x,z\cdot y].
		\end{align}
	\end{defn}
	
	\begin{defn}\label{tp-structures}
		Let $({\mathfrak L},[\cdot,\cdot])$ be a Lie algebra. A \textit{transposed Poisson (algebra) structure} on $({\mathfrak L},[\cdot,\cdot])$ is a commutative associative operation $\cdot$ on $\mathfrak L$ which makes $({\mathfrak L},\cdot,[\cdot,\cdot])$ a transposed Poisson algebra.
	\end{defn}
	
	A transposed Poisson structure $\cdot$ on $\LL$ is called \textit{trivial}, if $x\cdot y=0$ for all $x,y\in\LL$.
	
	\begin{defn}\label{12der}
		Let $({\mathfrak L}, [\cdot,\cdot])$ be an algebra and $\varphi:\mathfrak L\to\mathfrak L$ a linear map.
		Then $\varphi$ is a \textit{$\frac{1}{2}$-derivation} if it satisfies
		\begin{align}\label{vf(xy)=half(vf(x)y+xvf(y))}
			\varphi \big([x,y]\big)= \frac{1}{2} \big([\varphi(x),y]+ [x, \varphi(y)] \big).
		\end{align}
	\end{defn}
	Observe that $\frac{1}{2}$-derivations are a particular case of $\delta$-derivations introduced by Filippov in \cite{fil1}. The space of all $\frac{1}{2}$-derivations of an algebra $\mathfrak L$ will be denoted by $\Dl(\mathfrak L).$ It is easy to see from \cref{vf(xy)=half(vf(x)y+xvf(y))} that $[\LL,\LL]$ and $\Ann(\LL)$ are invariant under any $\frac 12$-derivation of $\LL$.

	\cref{tpa,12der} immediately imply the following key Lemma.
	\begin{lem}\label{glavlem}
		Let $({\mathfrak L},[\cdot,\cdot])$ be a Lie algebra and $\cdot$ a bilinear operation on ${\mathfrak L}$. Then $({\mathfrak L},\cdot,[\cdot,\cdot])$ is a transposed Poisson algebra 
		if and only if $\cdot$ is commutative and associative and for every $z\in{\mathfrak L}$ the multiplication by $z$ in $({\mathfrak L},\cdot)$ is a $\frac{1}{2}$-derivation of $({\mathfrak L}, [\cdot,\cdot]).$
	\end{lem}
	
	The basic example of a $\frac{1}{2}$-derivation is the multiplication by a field element.
	Such $\frac{1}{2}$-derivations will be called \textit{trivial}. 
	
	\begin{thrm}\label{princth}
		Let $\LL$ be a Lie algebra with $\dim(\LL)>1$. If all the $\frac{1}{2}$-derivations of $\LL$ are trivial, then all the transposed Poisson algebra structures on $\LL$ are trivial.
	\end{thrm}
	
	Another well-known class of $\frac{1}{2}$-derivations of $(\LL,[\cdot,\cdot])$ is formed by linear maps $\LL\to \Ann(\LL)$ annihilating $[\LL,\LL]$. If $(\LL,[\cdot,\cdot])$ is a Lie algebra, such $\frac{1}{2}$-derivations of $\LL$ (we call them \textit{central-valued}) correspond to the following transposed Poisson structures on $\LL$. Denote by $Z(\LL)$ the \textit{center} of $\LL$ and fix a complement $V$ of $[\LL,\LL]$ in $\LL$. Then any commutative operation $*:V\times V\to Z(\LL)$ extends to a commutative operation $\cdot:\LL\times\LL\to\LL$ by means of
	\begin{align}\label{(a_1+a_2)-times-(b_1+b_2)}
		(a_1+a_2)\cdot(b_1+b_2)=a_1*b_1,
	\end{align}
	where $a_1,b_1\in V$ and $a_2,b_2\in [\LL,\LL]$. The resulting product $\cdot$ satisfies \cref{Trans-Leibniz}. Indeed, the right-hand side of \cref{Trans-Leibniz} is zero, because $z\cdot x,z\cdot y\in Z(\LL)$, while the left-hand side of \cref{Trans-Leibniz} is zero by \cref{(a_1+a_2)-times-(b_1+b_2)}, because $[x,y]\in[\LL,\LL]$. As to the associativity of $\cdot$, consider the following two situations that often happen in practice. If $Z(\LL)\cap[\LL,\LL]=\{0\}$ and we choose $V\supseteq Z(\LL)$, then $\cdot$ is associative $\iff *$ is associative. If $Z(\LL)\sst [\LL,\LL]$, then $(a\cdot b)\cdot c=a\cdot(b\cdot c)=0$, so $\cdot$ is associative for any $*$. Observe that, whenever it is associative, $\cdot$ is at the same time a usual Poisson structure on $(\LL,[\cdot,\cdot])$. Thus, such transposed Poisson structures are said to be \textit{of Poisson type}.
	
	One more class of $\frac 12$-derivations that can be defined on any Lie algebra $(\LL,[\cdot,\cdot])$ is as follows. Fix $c\in Z([\LL,\LL])$ and let 
	\begin{align}\label{af_c(a)=[c_a]}
		\ad_c(a)=[c,a]
	\end{align}
for all $a\in\LL$. Then $[\ad_c(a),b]+[a,\ad_c(b)]=\ad_c([a,b])=0$, because $c\in Z([\LL,\LL])$. The $\frac 12$-derivation $\ad_c$ will be called \textit{inner}. It corresponds to the transposed Poisson structure given by the following \textit{mutation} of the product $[\cdot,\cdot]$:
\begin{align*}%\label{a-cdot_c-b=[[ac]b]}
	a\cdot_c b=[[a,c],b]
\end{align*}
for all $a,b\in\LL$. Indeed, $\cdot_c$ is commutative due to the Jacobi identity and $c\in Z([\LL,\LL])$. It is associative, because $(x\cdot_c y)\cdot_c z=x\cdot_c (y\cdot_c z)=0$. Since $Z([\LL,\LL])$ is an ideal of $\LL$ (as $\ad_{[a,c]}=[\ad_a,\ad_c]$), then the $\cdot_c$-multiplication by $a$, being $\ad_{[a,c]}$, is an inner $\frac 12$-derivation of $\LL$. We call $\cdot_c$ a \textit{mutational} structure on $\LL$.

Given two binary operations $\cdot_1$ and $\cdot_2$ on a vector space $V$, their \textit{sum} $*$ is defined by 
\begin{align}\label{a*b=a-cdot_1-b+a-cdot_2-b}
	a*b=a\cdot_1 b+a\cdot_2 b.
\end{align}
We say that $\cdot_1$ and $\cdot_2$ are \textit{orthogonal}, if
\begin{align*}
	V\cdot_1V\sst\Ann(V,\cdot_2)\text{ and }V\cdot_2V\sst\Ann(V,\cdot_1).
\end{align*}
In this case $*$ given by \cref{a*b=a-cdot_1-b+a-cdot_2-b} is called the \textit{orthogonal} sum of $\cdot_1$ and $\cdot_2$. 

Clearly, the sum $*$ of two transposed Poisson structures $\cdot_1$ and $\cdot_2$ on $(\LL,[\cdot,\cdot])$ is commutative and satisfies \cref{Trans-Leibniz}. If $\cdot_1$ and $\cdot_2$ are orthogonal, then $*$ is also associative, so we get the following.
\begin{prop}
	The orthogonal sum of two transposed Poisson structures on a Lie algebra $\LL$ is a transposed Poisson structure on $\LL$.
\end{prop}
Observe that any mutational transposed Poisson structure on a Lie algebra $\LL$ is orthogonal to any transposed Poisson structure of Poisson type on $\LL$. 
\begin{cor}\label{mutational+Poisson-TP}
	The (orthogonal) sum of a mutational transposed Poisson structure on $\LL$ and a transposed Poisson structure of Poisson type on $\LL$ is a transposed Poisson structure on $\LL$.
\end{cor}
	% Let us recall the definition of ${\rm Hom}$-structures on Lie algebras.
	% \begin{defn}
		% 	Let $({\mathfrak L}, [\cdot,\cdot])$ be a Lie algebra and $\varphi$ be a linear map.
		% 	Then $({\mathfrak L}, [\cdot,\cdot], \varphi)$ is a ${\rm Hom}$-Lie structure on $({\mathfrak L}, [\cdot,\cdot])$ if 
		% 	\[
		% 	[\varphi(x),[y,z]]+[\varphi(y),[z,y]]+[\varphi(z),[x,y]]=0.
		% 	\]
		% \end{defn}
	% Filippov proved that each nonzero $\delta$-derivation ($\delta\neq0,1$) of a Lie algebra gives a non-trivial ${\rm Hom}$-Lie algebra structure \cite[Theorem 1]{fil1}.	

	%	\subsection{Isomorphic transposed Poisson algebra structures}
	
	Let $\cdot$ be a transposed Poisson algebra structure on a Lie algebra $({\mathfrak L}, [\cdot,\cdot])$. 
	Then any automorphism $\phi$ of $({\mathfrak L}, [\cdot,\cdot])$ induces the transposed Poisson algebra structure $*$ on $({\mathfrak L}, [\cdot,\cdot])$ given by
	\begin{align*}%\label{x*y=phi(phi-inv(x)phi-inv(y))}
		x*y=\phi\big(\phi^{-1}(x)\cdot\phi^{-1}(y)\big),\ \ x,y\in{\mathfrak L}.
	\end{align*}
	Clearly, $\phi$ is an isomorphism of transposed Poisson algebras $({\mathfrak L},\cdot,[\cdot,\cdot])$ and $({\mathfrak L},*,[\cdot,\cdot])$.

	\subsection{Posets and incidence algebras}
	
	Let $(X,\le)$ be a finite poset. We say that $x,y\in X$ are \textit{comparable} if either $x\le y$ or $y\le x$. A \textit{chain} in $X$ is a subset $C\sst X$ such that any two $x,y\in C$ are comparable. A chain $C$ is \textit{maximal}, if it is not contained in any chain different from $C$. The \textit{length} of a chain $C\sst X$ is defined to be $l(C):=|C|-1$. The \textit{length} of $X$ is $l(X):=\max\{l(C)\mid C\text{ is a chain in }X\}$. For all $x\le y$ we write $l(x,y):=l(\{z\in X \mid x\leq z\leq y\})$. A {\it walk} in $X$ from $x$ to $y$ is an ordered sequence $\G:x=x_0,x_1,\dots,x_m=y$ of elements of $X$, such that $x_i$ is comparable with $x_{i+1}$ and $l(x_i,x_{i+1})=1$ (if $x_i\le x_{i+1}$) or $l(x_{i+1},x_i)=1$ (if $x_{i+1}\le x_i$) for all $i=0,\dots,m-1$. Each $x_i$, $0\le i\le m$, is a called a \textit{vertex} of $\G$, and each ordered pair $(x_i,x_{i+1})$, $0\le i\le m-1$, is called an \textit{edge} of $\G$. The walk $\G$ is said to be {\it closed} if $x_0=x_m$. A {\it path} is a walk with $x_i\ne x_j$ for $i\ne j$. A {\it cycle} is a closed walk $x_0,x_1,\dots,x_m=x_0$ such that $m\ge 4$ and $x_i=x_j\impl \{i,j\}=\{0,m\}$  for $i\ne j$. A poset $X$ is {\it connected} if for all $x,y\in X$ there is a path from $x$ to $y$. 
	
	We will denote by $\Min(X)$ (resp.~$\Max(X)$) the set of minimal (resp.~maximal) elements of $X$ and by $X^2_<$ the set of pairs $(x,y)\in X^2$ such that $x<y$.
	
	Let $X$ be a finite poset and $K$ a field. The \textit{incidence algebra} $I(X,K)$ of $X$ over $K$ (see~\cite{Rota64}) is the associative $K$-algebra with basis $\{e_{xy} \mid x\le y\}$
	and multiplication given by
	$$
	e_{xy}e_{uv}=
	\begin{cases}
		e_{xv}, & y=u,\\
		0, & y\ne u,
	\end{cases}
	$$
	for all $x\leq y$ and $u\leq v$ in $X$. Given $f\in I(X,K)$, we write $f=\sum_{x\le y}f(x,y)e_{xy}$, where $f(x,y)\in K$. Let us denote $e_x := e_{xx}$, and for arbitrary $Y\sst X$ put $e_Y:=\sum_{y\in Y}e_y$. Then $e_Y$ is an idempotent and $e_Ye_Z=e_{Y\cap Z}$, in particular, $e_xe_y=0$ for $x\ne y$. Notice that $\dl:=e_X$ is the identity element of $I(X,K)$. More generally,
	\begin{align*}
		e_Yfe_Z=\sum_{Y\ni y\le z\in Z}f(y,z)e_{yz}.
	\end{align*}
	
	As usual, we denote by $[f,g]=fg-gf$ the commutator product on $I(X,K)$, so $(I(X,K),[\cdot,\cdot])$ is a Lie algebra. If $X$ is connected, then one can easily prove that 
	\begin{align*}
		Z(I(X,K))=\gen{\dl}\text{ and }[I(X,K),I(X,K)]=\gen{e_{xy}\mid x<y}     
	\end{align*}
(for instance, see \cite[Corollary 1.3.15]{SpDo} and \cite[Proposition 2.3]{FKS}). Moreover, by \cite[Proposition 2.5]{FKS}
	\begin{align*}%\label{Z([I(X_K)_I(X_K)])=<basis>}
		Z([I(X,K),I(X,K)])=\gen{e_{xy}\mid \Min(X)\ni x<y\in\Max(X)}.
	\end{align*}
		\textit{Diagonal elements} of $I(X,K)$ are $f\in I(X,K)$ with $f(x,y)=0$ for $x\neq y$. They form a commutative subalgebra $D(X,K)$ of $I(X,K)$ with basis $\{e_x \mid x \in X\}$. As a vector space, 
		\begin{align*}
			I(X,K)=D(X,K)\oplus [I(X,K),I(X,K)].
		\end{align*}
	Thus, each $f\in I(X,K)$ is uniquely written as $f=f_D+f_J$ with $f_D\in D(X,K)$ and $f_J\in [I(X,K),I(X,K)]$. Observe that $Z(I(X,K))\sst D(X,K)$.
	
	\section{$\frac 12$-derivations of the incidence algebra}\label{sec-TP-on-T_n}
	
	Throughout the paper $X$ will always stand for a connected finite poset and $K$ for a field of characteristic zero. We denote by $\Dl(I(X,K))$ the space of $\frac 12$-derivations of the Lie algebra $(I(X,K),[\cdot,\cdot])$.
	
	\subsection{The decomposition of a $\frac{1}{2}$-derivation of the incidence algebra}

	Given $Y\sst X$, the algebra $I(Y,K)$ will be naturally seen as a subalgebra of $I(X,K)$. Denote by $I(Y,K)^c$ the subspace of $I(X,K)$ consisting of $f\in I(X,K)$ such that $f(x,y)=0$, whenever $x,y\in Y$. Then
	\begin{align*}
		I(X,K)=I(Y,K)\oplus I(Y,K)^c
	\end{align*}
	as $K$-spaces. Namely, each $f\in I(X,K)$ uniquely decomposes as
	\begin{align*}
		f=f_Y+f_Y^c,
	\end{align*}
where
	\begin{align*}
		f_Y(x,y)=
		\begin{cases}
			f(x,y), & x,y\in Y,\\
			0, & \text{otherwise},
		\end{cases}
	\end{align*}
	(and $f_Y^c=f-f_Y$).
	\begin{lem}
		We have $I(Y,K)\cdot I(Y,K)^c+I(Y,K)^c\cdot I(Y,K)\sst I(Y,K)^c$.
	\end{lem}
\begin{proof}
	Let $f\in I(Y,K)$ and $g\in I(Y,K)^c$. For all $x,y\in Y$ and $x\le z\le y$ either $z\in Y$, in which case $g(z,y)=0$, or $z\not\in Y$, in which case $f(x,z)=0$, so $f(x,z)g(z,y)=0$. Therefore, $(fg)(x,y)=\sum_{x\le z\le y}f(x,z)g(z,y)=0$ for all $x,y\in Y$. This proves $I(Y,K)\cdot I(Y,K)^c\sst I(Y,K)^c$. The proof that $I(Y,K)^c\cdot I(Y,K)\sst I(Y,K)^c$ is analogous.
\end{proof}

	\begin{cor}\label{[I(Y_K)_ol(I(Y_K))]-sst-ol(I(Y_K))}
		It follows that $[I(Y,K),I(Y,K)^c]\sst I(Y,K)^c$.
	\end{cor}

	Given a linear map $I(X,K)\to I(X,K)$, denote by $\vf_Y$ the associated linear map $I(Y,K)\to I(Y,K)$ that sends $f\in I(Y,K)$ to
	\begin{align*}
		\vf_Y(f)=\vf(f)_Y.
	\end{align*}

	\begin{lem}\label{vf_Y-dl-der}
		Let $\vf\in \Dl(I(X,K))$. Then $\vf_Y\in \Dl(I(Y,K))$.
	\end{lem}
	\begin{proof}
	For all $f,g\in I(Y,K)$ we have
\begin{align*}
	\vf([f,g])&=\frac 12([\vf(f),g]+[f,\vf(g)])=\frac 12([\vf_Y(f)+\vf(f)_Y^c,g]+[f,\vf_Y(g)+\vf(g)_Y^c])\\
	&=\frac 12([\vf_Y(f),g]+[f,\vf_Y(g)])+\frac 12([\vf(f)_Y^c,g]+[f,\vf(g)_Y^c]),
\end{align*}
where $\frac 12([\vf_Y(f),g]+[f,\vf_Y(g)])\in I(Y,K)$ and $\frac 12([\vf(f)_Y^c,g]+[f,\vf(g)_Y^c])\in I(Y,K)^c$ by \cref{[I(Y_K)_ol(I(Y_K))]-sst-ol(I(Y_K))}. Hence,
$
	\vf_Y([f,g])=\frac 12([\vf_Y(f),g]+[f,\vf_Y(g)]).
$
	\end{proof}

\begin{lem}
	Let $\vf\in\Dl(I(X,K))$ and $x<y$. Then 
	\begin{align}
		[\vf(e_x),e_{xy}]&=(\vf(e_x)(x,x)-\vf(e_x)(y,y))e_{xy},\label{[vf(e_x)_e_xy]=(vf(e_x)(x_x)-vf(e_x)(y_y))e_xy}\\
		[e_{xy},\vf(e_y)]&=(\vf(e_y)(y,y)-\vf(e_y)(x,x))e_{xy}.\label{[e_xy_vf(e_y)]=(vf(e_y)(y_y)-vf(e_y)(x_x))e_xy}
	\end{align}
\end{lem}	
\begin{proof}
	Take an arbitrary $u<x$ and let $Y=\{u,x,y\}\sst X$. Then $\vf_Y\in\Dl(I(Y,K))$ by \cref{vf_Y-dl-der}. Observe that the isomorphism of chains $\{u,x,y\}\sst X$ and $\{1,2,3\}\sst\mathbb{N}$ induces an isomorphism of algebras $I(Y,K)\cong T_3(K)$, so by \cite[Proposition 10]{KK7} we have $\vf_Y(e_x)(u,x)=0$. Hence, $\vf(e_x)(u,x)=\vf(e_x)_Y(u,x)=\vf_Y(e_x)(u,x)=0$. It follows that
	\begin{align}\label{[vf(e_x)_e_xy](u_y)=0}
		[\vf(e_x),e_{xy}](u,y)=0,\text{ for all }u<x<y.
	\end{align}
Similarly, for any $v>y$ consider $Y=\{x,y,v\}\sst X$. We again have $I(Y,K)\cong T_3(K)$, and $\vf_Y(e_x)(y,v)=0$ by \cite[Proposition 10]{KK7}. Since $\vf_Y(e_x)(y,v)=\vf(e_x)_Y(y,v)=\vf(e_x)(y,v)$,
\begin{align}\label{[vf(e_x)_e_xy](x_v)=0}
	[\vf(e_x),e_{xy}](x,v)=0,\text{ for all }x<y<v.
\end{align}
Combining \cref{[vf(e_x)_e_xy](u_y)=0,[vf(e_x)_e_xy](x_v)=0}, we conclude that
\begin{align*}
	[\vf(e_x),e_{xy}]=[\vf(e_x),e_{xy}](x,y)e_{xy}=(\vf(e_x)(x,x)-\vf(e_x)(y,y))e_{xy},
\end{align*}
whence \cref{[vf(e_x)_e_xy]=(vf(e_x)(x_x)-vf(e_x)(y_y))e_xy}. The proof of \cref{[e_xy_vf(e_y)]=(vf(e_y)(y_y)-vf(e_y)(x_x))e_xy} is similar.
\end{proof}

\begin{lem}\label{vf(e_xy)=multiple-of-e_xy}
		Let $\vf\in\Dl(I(X,K))$ and $x<y$. Then 
	\begin{align}\label{vf(e_xy)=(vf(e_y)(y_y)-vf(e_y)(x_x))e_xy}
		\vf(e_{xy})=(\vf(e_x)(x,x)-\vf(e_x)(y,y))e_{xy}=(\vf(e_y)(y,y)-\vf(e_y)(x,x))e_{xy}.
	\end{align}
\end{lem}
\begin{proof}
Let $u\le x$ and $Y=\{u,x,y\}$. Then $\vf(e_{xy})(u,x)=\vf_Y(e_{xy})(u,x)=0$, where the latter follows from \cref{vf_Y-dl-der} and \cite[Proposition 10]{KK7}. Hence, 
\begin{align}\label{vf(e_xy)e_x=0}
	\vf(e_{xy})e_x=0.
\end{align}
Now, applying $\vf$ to $e_{xy}=[e_x,e_{xy}]$ and using \cref{[vf(e_x)_e_xy]=(vf(e_x)(x_x)-vf(e_x)(y_y))e_xy,vf(e_xy)e_x=0}, we have
\begin{align}\label{2vf(e_xy)=(vf(e_x)(x_x)-vf(e_x)(y_y))e_xy+e_xvf(e_xy)}
	2\vf(e_{xy})=(\vf(e_x)(x,x)-\vf(e_x)(y,y))e_{xy} + e_x\vf(e_{xy}).
\end{align}
The rest of the proof is very similar to that of \cite[Lemma 7]{KK7}. Multiplying \cref{2vf(e_xy)=(vf(e_x)(x_x)-vf(e_x)(y_y))e_xy+e_xvf(e_xy)} by $e_x$ on the left, we get
\begin{align*}
	e_x\vf(e_{xy})=(\vf(e_x)(x,x)-\vf(e_x)(y,y))e_{xy},
\end{align*}
so \cref{2vf(e_xy)=(vf(e_x)(x_x)-vf(e_x)(y_y))e_xy+e_xvf(e_xy)} implies the first equality of \cref{vf(e_xy)=(vf(e_y)(y_y)-vf(e_y)(x_x))e_xy} in view of $\ch (K)\ne 2$. The second one is proved applying $\vf$ to $e_{xy}=[e_{xy},e_y]$ and using \cref{[e_xy_vf(e_y)]=(vf(e_y)(y_y)-vf(e_y)(x_x))e_xy}. 
\end{proof}

The next result is a generalization of \cite[Lemma 11]{KK7}.
\begin{lem}\label{vf(e_xy)(x_y)=const}
	Let $\vf\in\Dl(T_n(F))$. 
	\begin{enumerate}
		\item\label{vf(e_xx)(u_u)=vf(e_xx)(v_v)} For all $x\in X$ and $u\le v$ with $x\not\in\{u,v\}$ we have $\vf(e_x)(u,u)=\vf(e_x)(v,v)$.
		\item\label{vf(e_xy)(x_y)=vf(e_uv)(u_v)} For all $x<y$ and $u<v$ belonging to a same chain we have $\vf(e_{xy})(x,y)=\vf(e_{uv})(u,v)$.
		\item\label{vf(e_xx)(u_v)=0} For all $x\in X$ and $u<v$ we have $\vf(e_x)(u,v)=0$, unless $\Min(X)\ni x=u<v\in\Max(X)$ or $\Min(X)\ni u<v=x\in\Max(X)$.
		\item\label{vf(e_xx)(x_y)=-vf(e_yy)(x_y)} For all $x\in \Min(X)$ and $y\in\Max(X)$ we have $\vf(e_x)(x,y)=-\vf(e_y)(x,y)$.
	\end{enumerate}
\end{lem}
\begin{proof}
	\textit{\cref{vf(e_xx)(u_u)=vf(e_xx)(v_v)}}. The proof is totally analogous to that of \cite[Lemma 11 (i)]{KK7}: apply $\vf$ to $[e_x,e_{uv}]=0$ and multiply the resulting equality by $e_u$ on the left and by $e_v$ on the right.
	
	\textit{\cref{vf(e_xy)(x_y)=vf(e_uv)(u_v)}}. Observe that $Y=\{x,y,u,v\}$ is a chain and denote $a=\Min(Y)$ and $b=\Max(Y)$. Then by \cref{vf_Y-dl-der} and \cite[Lemma 11 (ii)]{KK7} we have
	\begin{align*}
		\vf(e_{xy})(x,y)=\vf_Y(e_{xy})(x,y)=\vf_Y(e_{ab})(a,b)=\vf_Y(e_{uv})(u,v)=\vf(e_{uv})(u,v).
	\end{align*}

	\textit{\cref{vf(e_xx)(u_v)=0}}. If $x\not\in\{u,v\}$, then $\vf(e_x)(u,v)=0$ follows by applying $\vf$ to $[e_x,e_u]=0$ and multiplying the resulting equality by $e_u$ on the left and by $e_v$ on the right. Let $u=x$. If $u\not\in\Min(X)$, then choose $t<u$ and set $Y=\{t,x,v\}$. If $v\not\in\Max(X)$, then choose $w>v$ and set $Y=\{x,v,w\}$. Thanks to \cref{vf_Y-dl-der} and \cite[Proposition 10]{KK7} we have $\vf(e_x)(u,v)=\vf_Y(e_x)(x,v)=0$. The case $v=x$ is similar.
	
	\textit{\cref{vf(e_xx)(x_y)=-vf(e_yy)(x_y)}}. Set $Y=\{x,y\}$. Then $\vf(e_x)(x,y)=\vf_Y(e_x)(x,y)=-\vf_Y(e_y)(x,y)=-\vf(e_y)(x,y)$ by \cref{vf_Y-dl-der} and \cite[Proposition 10]{KK7}.
\end{proof}

\begin{cor}\label{vf(D)-sst-Z<=>vf([I_I])=0}
	Let $\vf\in\Dl(I(X,K))$. Then 
	\begin{align}\label{vf(D(X_K))-sst-Z(I(X_K))-iff-vf([I(X_K)_I(X_K)])=0}
		\vf(D(X,K))\sst Z(I(X,K))\iff \vf([I(X,K),I(X,K)])=\{0\}.
	\end{align}
\end{cor}
\begin{proof}
	If $\vf(e_x)\in Z(I(X,K))$, then $\vf(e_{xy})=0$ for all $x<y$ by \cref{vf(e_xy)=multiple-of-e_xy}, proving the ``$\impl$''-part of \cref{vf(D(X_K))-sst-Z(I(X_K))-iff-vf([I(X_K)_I(X_K)])=0}. Conversely, if $\vf(e_{ux})=\vf(e_{xv})=0$, then $\vf(e_x)(u,u)=\vf(e_x)(x,x)=\vf(e_x)(v,v)$ for all $u<x<v$ by \cref{vf(e_xy)=multiple-of-e_xy}. Since by \cref{vf(e_xy)(x_y)=const}\cref{vf(e_xx)(u_u)=vf(e_xx)(v_v)} we always have $\vf(e_x)(u,u)=\vf(e_x)(v,v)$, whenever $x\not\in\{u,v\}$ and $u<v$, then $\vf(e_x)\in Z(I(X,K))$ by connectedness of $X$.
\end{proof}

\begin{lem}\label{ad_c-is-halfder}
	Let $\vf\in\Dl(I(X,K))$ and define $c\in Z([I(X,K),I(X,K)])$ by
	\begin{align}\label{c(x_y)=vf(e_y)(x_y)}
		c(x,y)=
		\begin{cases}
			\vf(e_y)(x,y), & \Min(X)\ni x<y\in\Max(X),\\
			0, & \text{otherwise}.
		\end{cases}
	\end{align}
	Then for all $x<y$
	\begin{align}\label{af_c(e_xy)-and-af_c(e_x)}
		\ad_c(e_{xy})=0\text{ and }\ad_c(e_x)=\vf(e_x)_J.
	\end{align}
\end{lem}
\begin{proof}
	The first equality of \cref{af_c(e_xy)-and-af_c(e_x)} is obvious in view of \cref{af_c(a)=[c_a]}, because $e_{xy}\in [I(X,K),I(X,K)]$. Let $x\in X$ and $u<v$. Then
	\begin{align*}
		\ad_c(e_x)(u,v)=[c,e_x](u,v)=
		\begin{cases}
			-c(u,v), & x=u,\\
			c(u,v), & x=v,\\
			0, & x\not\in\{u,v\}.
		\end{cases}
	=
	\begin{cases}
		-\vf(e_v)(u,v), & x=u,\\
		\vf(e_x)(u,v), & x=v,\\
		0, & x\not\in\{u,v\}.
	\end{cases}
	\end{align*}
We immediately see that $\ad_c(e_x)(u,v)=\vf(e_x)(u,v)$, whenever $x=v\in\Max(X)$ and $u\in\Min(X)$. If $x=u\in\Min(X)$ and $v\in\Max(X)$, then $\ad_c(e_x)(u,v)=-\vf(e_v)(u,v)=\vf(e_x)(u,v)$ by \cref{vf(e_xy)(x_y)=const}\cref{vf(e_xx)(x_y)=-vf(e_yy)(x_y)}. In the remaining cases $\ad_c(e_x)(u,v)=0=\vf(e_x)(u,v)$ by \cref{vf(e_xy)(x_y)=const}\cref{vf(e_xx)(u_v)=0}. Thus, the second equality of \cref{af_c(e_xy)-and-af_c(e_x)} is established.
\end{proof} 

\begin{defn}
	Let $\vf:I(X,K)\to I(X,K)$ be a linear map. We say that $\vf$ is {\it diagonality preserving}, if $\vf(D(X,K))\sst D(X,K)$.
\end{defn}

\begin{rem}\label{vf-on-D(X_K)}
	Any diagonality preserving linear map $\vf:I(X,K)\to I(X,K)$ is a $\frac 12$-derivation, when restricted to $D(X,K)$.
\end{rem}

\begin{prop}\label{vf=inner+diag-pres}
	Each $\vf\in\Dl(I(X,K))$ is of the form
	\begin{align}\label{vf=ad_c+tl-vf}-
		\vf=\ad_c+\tl\vf
	\end{align}
for a unique $c\in Z([I(X,K),I(X,K)])$ and a unique diagonality preserving $\tl\vf\in\Dl(I(X,K))$.
\end{prop}
\begin{proof}
	Consider $c\in Z([I(X,K),I(X,K)])$ given by \cref{c(x_y)=vf(e_y)(x_y)} and set $\tl\vf:=\vf-\ad_c$. Then $\tl\vf\in\Dl(I(X,K))$, and \cref{vf=ad_c+tl-vf} holds by definition. Moreover, $\tl\vf$ is diagonality preserving thanks to \cref{af_c(e_xy)-and-af_c(e_x)}. 
	
	For the uniqueness of $c$ and $\tl\vf$, assume the decomposition \cref{vf=ad_c+tl-vf}, where $c\in Z([I(X,K),I(X,K)])$ and $\tl\vf$ is diagonality preserving. Then
	\begin{align*}%\label{vf|_[I_I]=tl-vf|_[I_I]}
		\tl\vf(e_{xy})=\vf(e_{xy})\text{ and }\tl\vf(e_x)=\vf(e_x)_D
	\end{align*}
for all $x<y$, because $\ad_c(e_{xy})=0$ and $\tl\vf(e_x)_J=\ad_c(e_x)_D=0$. Hence, $\tl\vf$ is unique. Since $c(x,y)=\ad_c(e_y)(x,y)=(\vf-\tl\vf)(e_y)(x,y)$, then $c$ is also unique.
\end{proof}

\subsection{Diagonality preserving $\frac 12$-derivations of the incidence algebra}
Our next goal is to describe diagonality preserving $\frac 12$-derivations of $I(X,K)$. We begin with some facts that are true for an arbitrary $\vf\in\Dl(I(X,K))$. We know from \cref{vf(e_xy)=multiple-of-e_xy} that there exists $\sg=\sg_\vf:X^2_<\to K$ such that
\begin{align}\label{vf(e_xy)=sg(x_y)e_xy}
	\vf(e_{xy})=\sg(x,y)e_{xy}\text{ for all }x<y.
\end{align}
Moreover, by \cref{vf(e_xy)(x_y)=const}\cref{vf(e_xy)(x_y)=vf(e_uv)(u_v)} the map $\sg$ is \textit{constant on chains} in $X$ in the sense that
\begin{align*}%\label{sg-const-on-chains}
	\sg(x,y)=\sg(u,v),\text{ if }x<y\text{ and }u<v\text{ belong to a same chain in }X.
\end{align*}

\begin{rem}\label{vf-on-[I(X_K)_I(X_K)]}
	A linear map $\vf:I(X,K)\to I(X,K)$ satisfying \cref{vf(e_xy)=sg(x_y)e_xy} for some constant on chains map $\sg:X^2_<\to K$ is a $\frac 12$-derivation, when restricted to $[I(X,K),I(X,K)]$.
\end{rem}

The following example shows that $\sg$ can take different values on distinct maximal chains of $X$.
\begin{exm}\label{exm-2-chains}
	Let $X=\{1,2,3,4,5\}$ with the following Hasse diagram.
	
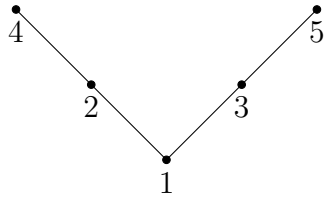
\begin{center}
		\begin{tikzpicture}
			\draw  (0,0)--(-1,1);
			\draw  (-1,1)--(-2,2);
			\draw  (0,0)-- (1,1);
			\draw  (1,1)-- (2,2);
			\draw [fill=black] (0,0) circle (0.05);
			\draw  (0,-0.3) node {$1$};
			\draw [fill=black] (-1,1) circle (0.05);
			\draw  (-1,0.7) node {$2$};
			\draw [fill=black] (-2,2) circle (0.05);
			\draw  (-2,1.7) node {$4$};
			\draw [fill=black] (1,1) circle (0.05);
			\draw  (1,0.7) node {$3$};
			\draw [fill=black] (2,2) circle (0.05);
			\draw  (2,1.7) node {$5$};
		\end{tikzpicture}
	\end{center}
Consider the $K$-linear map $\vf:I(X,K)\to I(X,K)$ defined as follows: 
$\vf(e_1)=e_1+e_3+e_5$,
$\vf(e_{12})=e_{12}$,
$\vf(e_{14})=e_{14}$,
$\vf(e_2)=e_2$,
$\vf(e_{24})=e_{24}$,
$\vf(e_4)=e_4$,
$\vf(e_{13})=\vf(e_{15})=\vf(e_3)=\vf(e_{35})=\vf(e_5)=0$.
Then $\vf\in\Dl(I(X,K))$ and on $[I(X,K),I(X,K)]$ it acts by \cref{vf(e_xy)=sg(x_y)e_xy}, where $\sg(1,2)=\sg(1,4)=\sg(2,4)=1$ and $\sg(1,3)=\sg(1,5)=\sg(3,5)=0$. It is also diagonality preserving. 
%In view of \cref{vf-on-D(X_K),vf-on-[I(X_K)_I(X_K)]} it remains to analyze the action of $\vf$ on the products $[e_x,e_{uv}]$, where $u<v$. 
%The case $x,u,v\in\{1,2,4\}$ is immediate, because on $I(\{1,2,4\},K)$ the map $\vf$ acts as the identity plus a map with values in the centralizer of $I(\{1,2,4\},K)$. The case $x,u,v\in\{1,3,5\}$ is trivial, because $\vf(I(\{1,3,5\},K))\sst Z(I(\{1,3,5\},K))$. Since $\vf(I(\{2,4\},K))=I(\{2,4\},K)$, $\vf(I(\{3,5\},K))=\{0\}$ and $[I(\{2,4\},K),I(\{1,3,5\},K)]=[I(\{3,5\},K),I(\{1,2,4\},K)]=\{0\}$, we are done.
\end{exm}

We are going to find one more condition that $\sg$ must satisfy. To this end, summarize the results of \cref{vf(e_xy)=multiple-of-e_xy} and \cref{vf(e_xy)(x_y)=const}\cref{vf(e_xx)(u_u)=vf(e_xx)(v_v)} into the following
\begin{cor}
	Let $\vf\in\Dl(I(X,K))$. Then for all $x\in X$ and $u<v$ in $X$:
	\begin{align}\label{vf(e_x)(u_u)-vf(e_x)(v_v)-in-terms-of-sg}
		\vf(e_x)(u,u)-\vf(e_x)(v,v)=
		\begin{cases}
			\sg(x,v), & x=u,\\
			-\sg(u,x), & x=v,\\
			0, & x\not\in\{u,v\}.
		\end{cases}
	\end{align}
\end{cor}

\begin{defn}
	Let $\sg:X^2_<\to K$ be a map and $\G: u_0,u_1,\dots,u_m$ a walk in $X$. Define the following $4$ functions $s^\pm_{\sg,\G},t^\pm_{\sg,\G}:X\to K$:
	\begin{align}
		s^+_{\sg,\G}(x)&=\sum_{x=u_i<u_{i+1}} \sg(x,u_{i+1}),\ 
		s^-_{\sg,\G}(x)=\sum_{u_i>u_{i+1}=x} \sg(x,u_i), \label{s^+-and-s^-}\\ 
		t^+_{\sg,\G}(x)&=\sum_{x=u_i>u_{i+1}} \sg(u_{i+1},x), \ 
		t^-_{\sg,\G}(x)=\sum_{u_i<u_{i+1}=x} \sg(u_i,x),\label{t^+-and-t^-}
	\end{align}
where in all the sums above $0\le i\le m-1$.
\end{defn}

\begin{defn}
	Let $\G:u_0,u_1,\dots,u_m$ and $\G':u'_0,u'_1,\dots,u'_{m'}$ be two walks in $X$. We say that $\G$ and $\G'$ are \textit{composable}, if $u_m=u'_0$, in which case their \textit{composition} is defined to be the walk $\G*\G':u_0,u_1,\dots,u_m=u'_0,u'_1,\dots,u'_{m'}$. The \textit{inverse} of $\G$ is the walk $\G\m:u_m,u_{m-1},\dots,u_0$.
\end{defn}

\begin{lem}\label{properties-of-s-and-t}
	Let $\sg:X^2_<\to K$ be a map and $\G$, $\G'$ two composable walks in $X$.% Then for any $x\in X$
	\begin{enumerate}
		\item\label{s^+_sg_G*G'=s^+_sg_G+s^+_sg_G'} $s^\pm_{\sg,\G*\G'}=s^\pm_{\sg,\G}+s^\pm_{\sg,\G'}$ and $t^\pm_{\sg,\G*\G'}=t^\pm_{\sg,\G}+t^\pm_{\sg,\G'}$;
		\item\label{s^+_sg_G^(-1)=-s^+_sg_G} $s^\pm_{\sg,\G\m}=s^\mp_{\sg,\G}$ and $t^\pm_{\sg,\G\m}=t^\mp_{\sg,\G}$.
	\end{enumerate}
\end{lem}
\begin{proof}
	\textit{\cref{s^+_sg_G*G'=s^+_sg_G+s^+_sg_G'}.}
	Let $\G:u_0,u_1,\dots,u_m$, $\G':u'_0,u'_1,\dots,u'_{m'}$ and $\G*\G':u_0,u_1,\dots,u_{m+m'}$, where $u_{m+i}=u'_i$, $0\le i\le m'$. Then for all $x\in X$
	\begin{align*}
		&\{0\le i\le m+m'-1\mid x=u_i<u_{i+1}\}\\
		&\quad=\{0\le i\le m-1\mid x=u_i<u_{i+1}\}\sqcup\{m\le i\le m+m'-1\mid x=u_i<u_{i+1}\}\\
		&\quad=\{0\le i\le m-1\mid x=u_i<u_{i+1}\}\sqcup\{0\le i\le m'-1\mid x=u_{m+i}<u_{m+i+1}\}\\
		&\quad=\{0\le i\le m-1\mid x=u_i<u_{i+1}\}\sqcup\{0\le i\le m'-1\mid x=u'_i<u'_{i+1}\},
	\end{align*} 
showing that $s^+_{\sg,\G*\G'}(x)=s^+_{\sg,\G}(x)+s^+_{\sg,\G'}(x)$. The remaining $3$ equalities are proved analogously,

\textit{\cref{s^+_sg_G^(-1)=-s^+_sg_G}.} If $\G:u_0,u_1,\dots,u_m$, then $\G\m:u'_0,u'_1,\dots,u'_m$ with $u'_i=u_{m-i}$ for all $0\le i\le m$. Denote $i'=m-i-1$. Then
\begin{align*}
	\{0\le i\le m-1\mid x=u'_i<u'_{i+1}\}&=\{0\le i\le m-1\mid x=u_{m-i}<u_{m-i-1}\}\\
	&=\{0\le i'\le m-1\mid x=u_{i'+1}<u_{i'}\}
\end{align*}
and $\sg(x,u'_{i+1})=\sg(x,u_{i'})$, whence $s^+_{\sg,\G\m}(x)=s^-_{\sg,\G}(x)$. The remaining $3$ equalities are proved using the same argument.
\end{proof}

\begin{lem}\label{vf(e_x)(u_m_u_m)-in-terms-of-vf(e_x)(u_0_u_0)}
	Let $\vf\in\Dl(I(X,K))$, $x\in X$ and $\G:u_0,u_1,\dots,u_m$ a walk in $X$. Then
	\begin{align}\label{vf(e_x)(u_m_u_m)=vf(e_x)(u_0_u_0)+s-and-t}
		\vf(e_x)(u_m,u_m)=\vf(e_x)(u_0,u_0)-s^+_{\sg,\G}(x)+s^-_{\sg,\G}(x)-t^+_{\sg,\G}(x)+t^-_{\sg,\G}(x).
	\end{align}
\end{lem}
\begin{proof}
	Obviously,
	\begin{align}\label{vf(e_x)(u_m_u_m)=vf(e_x)(u_0_u_0)+sum-Phi}
		\vf(e_x)(u_m,u_m)=\vf(e_x)(u_0,u_0)+\sum_{i=0}^{m-1}\Phi_{\G,i}(x),
	\end{align}
	where $\Phi_{\G,i}(x)=\vf(e_x)(u_{i+1},u_{i+1})-\vf(e_x)(u_i,u_i)$ for all $0\le i\le m-1$. However, by \cref{vf(e_x)(u_u)-vf(e_x)(v_v)-in-terms-of-sg} we have
	\begin{align}\label{vf(e_x)(u_(i+1)_u_(i+1))-vf(e_x)(u_i_u_i)}
		\Phi_{\G,i}(x)=
		\begin{cases}
			-\sg(x,u_{i+1}), & x=u_i<u_{i+1},\\
			\sg(u_i,x), & u_i<u_{i+1}=x,\\
			-\sg(u_{i+1},x), & x=u_i>u_{i+1},\\
			\sg(x,u_i), & u_i>u_{i+1}=x,\\
			0, & x\not\in\{u_i,u_{i+1}\}.
		\end{cases}
	\end{align}
	Thus, \cref{vf(e_x)(u_m_u_m)=vf(e_x)(u_0_u_0)+s-and-t} follows by \cref{s^+-and-s^-,t^+-and-t^-,vf(e_x)(u_m_u_m)=vf(e_x)(u_0_u_0)+sum-Phi,vf(e_x)(u_(i+1)_u_(i+1))-vf(e_x)(u_i_u_i)}.
\end{proof}

\begin{cor}\label{uniqueness-of-vf_sg_k}
	Fix an arbitrary $u_0\in X$. Given two diagonality preserving $\vf,\psi\in\Dl(I(X,K))$ and $x\in X$,
	if $\vf(e_x)(u_0,u_0)=\psi(e_x)(u_0,u_0)$ and $\sg_\vf=\sg_\psi$, then $\vf(e_x)=\psi(e_x)$.
\end{cor}
\begin{proof}
	Let $v\in X$ and choose a walk $\G: u_0,u_1,\dots,u_m=v$ from $u_0$ to $v$. Then $\vf(e_x)(v,v)$ is given by the right-hand side of \cref{vf(e_x)(u_m_u_m)=vf(e_x)(u_0_u_0)+sum-Phi}, which is completely determined by $\vf(e_x)(u_0,u_0)$ and $\sg$. Choosing the same $\G$ for $\psi$, we have $\vf(e_x)(v,v)=\psi(e_x)(v,v)$. Thus, $\vf(e_x)=\psi(e_x)$, because $\vf(e_x),\psi(e_x)\in D(X,K)$.
\end{proof}

\begin{cor}\label{admissibility-of-sg}
	Let $\vf\in\Dl(I(X,K))$ and $\sg:X^2_<\to K$ the associated map given by \cref{vf(e_xy)=sg(x_y)e_xy}. Then for any closed walk $\G$ in $X$:
	\begin{align}\label{s^+-s^-+t^+-t^-=0}
		s^+_{\sg,\G}-s^-_{\sg,\G}+t^+_{\sg,\G}-t^-_{\sg,\G}=0.
	\end{align}
\end{cor}
\begin{proof}
	This follows from \cref{vf(e_x)(u_m_u_m)=vf(e_x)(u_0_u_0)+s-and-t}, because in a closed walk $\G:u_0,u_1,\dots,u_m$ one has $u_m=u_0$.
\end{proof}

\begin{defn}\label{defn-admissible}
	Let $\sg:X^2_<\to K$ be a map. We say that $\sg$ is \textit{admissible} if \cref{s^+-s^-+t^+-t^-=0} holds for any closed walk $\G$ in $X$. 
\end{defn}

Although in our proofs we will need the general \cref{defn-admissible}, in practice it is convenient to check \cref{s^+-s^-+t^+-t^-=0} on cycles. The following lemma is an analog of \cite[Lemma 5.13]{FKS}.

\begin{lem}
	A map $\sg:X^2_<\to K$ is admissible if and only if \cref{s^+-s^-+t^+-t^-=0} holds for any cycle $\G$ in $X$.
\end{lem}
\begin{proof}
	We only need to prove the ``if'' part. The proof is similar to that of \cite[Lemma 5.13]{FKS}. Assume that \cref{s^+-s^-+t^+-t^-=0} holds for any cycle $\G$ in $X$. Let now $\G:u_0,u_1,\dots,u_m=u_0$ be a closed walk in $X$. As in the proof of \cite[Lemma 5.13]{FKS} we call a pair of indices $(i,j)$, $i<j-1$, a \textit{repetition} if $u_i = u_j$.  The proof will be by induction on the number of repetitions in $\G$. 
	
	\textit{The base case.} If $\G$ has only one repetition $(0,m)$ with $m \ge 4$, then $\G$ is a cycle, so \cref{s^+-s^-+t^+-t^-=0} holds by assumption. The case $m=3$ is impossible, because there are no cycles of $3$ vertices. If $m=2$, then $\G$ is of the form $u<v>u$ or $u>v<u$. Consider the case $\G:u<v>u$. If $x=u$, then $s^{\pm}_{\sg,\G}(x)=\sg(x,v)$ and $t^{\pm}_{\sg,\G}(x)=0$; and if $x=v$, then $s^{\pm}_{\sg,\G}(x)=0$ and $t^{\pm}_{\sg,\G}(x)=\sg(u,x)$; otherwise $s^{\pm}_{\sg,\G}(x)=t^{\pm}_{\sg,\G}(x)=0$. The case $\G:u>v<u$ is treated analogously.
	
	\textit{The induction step.} Take $\G$ with a repetition $(k,l)\ne(0,m)$, $i<j-1$. If $k>0$ and $l<m$, then consider $\G':u_0,\dots,u_k,u_{l+1},\dots,u_m$ and $\G'':u_k,\dots,u_l$. These are closed walks whose numbers of repetitions are strictly less than that of $\G$. Hence, by induction hypothesis, $\G'$ and $\G''$ satisfy \cref{s^+-s^-+t^+-t^-=0}. Observe that $\G'=\G'_1*\G'_2$ and $\G=\G'_1*\G''*\G'_2$, where $\G'_1:u_0,\dots,u_k$ and $\G'_2:u_l,\dots,u_m$. Therefore, $s^\pm_{\sg,\G}=s^\pm_{\sg,\G'_1}+s^\pm_{\sg,\G''}+s^\pm_{\sg,\G'_2}=s^\pm_{\sg,\G'}+s^\pm_{\sg,\G''}$ by \cref{properties-of-s-and-t}\cref{s^+_sg_G*G'=s^+_sg_G+s^+_sg_G'}. The same way one proves that $t^{\pm}_{\sg,\G}=t^{\pm}_{\sg,\G'}+t^{\pm}_{\sg,\G''}$. Thus, \cref{s^+-s^-+t^+-t^-=0} for $\G$ follows from \cref{s^+-s^-+t^+-t^-=0} for $\G'$ and $\G''$. If $k=0$, then $l<m-1$ and $\G=\G'*\G''$ with $\G':u_0,\dots,u_l$ and $\G'':u_l,\dots,u_m$. The case $l=m$ and $k>1$ is analogous.
\end{proof}

\begin{defn}
	Let $\G:u_0,u_1,\dots,u_m=u_0$ be a cycle in $X$. We say that a map $\sg:X^2_<\to K$ is \textit{constant on $\G$}, if there exists $k\in K$ such that for all $0\le i\le m-1$ we have $\sg(u_i,u_{i+1})=k$, whenever $u_i<u_{i+1}$, and $\sg(u_{i+1},u_i)=k$, whenever $u_i>u_{i+1}$.
	If $\sg$ is constant on $\G$ for any cycle $\G$ in $X$, we say that $\sg$ is \textit{constant on cycles} in $X$.
\end{defn}

\begin{lem}\label{sg-admissible<=>constant-on-cycles}
	A map $\sg:X^2_<\to K$ is admissible if and only if $\sg$ is constant on cycles in $X$.
\end{lem}
\begin{proof}
	Let $\sg$ be admissible and $\G:u_0,u_1,\dots,u_m=u_0$ a cycle in $X$. Fix an arbitrary $1\le i\le m-1$. Since $\G$ is a cycle, for all $0\le j\le m-1$ we have $u_j=u_i\iff j=i$ and $u_{j+1}=u_i\iff j=i-1$ (modulo $m$). There are $4$ cases.
	
	\textit{Case 1.} $u_{i-1}<u_i<u_{i+1}$. Then $s^+_{\sg,\G}(u_i)=\sg(u_i,u_{i+1})$, $s^-_{\sg,\G}(u_i)=t^+_{\sg,\G}(u_i)=0$, $t^-_{\sg,\G}(u_i)=\sg(u_{i-1},u_i)$. Thus, \cref{s^+-s^-+t^+-t^-=0} yields 
	\begin{align}\label{sg(u_i_u_(i+1))=sg(u_(i-1)_u_i)}
		\sg(u_i,u_{i+1})=\sg(u_{i-1},u_i).
	\end{align}

	\textit{Case 2.} $u_{i-1}<u_i>u_{i+1}$. Then $s^\pm_{\sg,\G}(u_i)=0$, $t^+_{\sg,\G}(u_i)=\sg(u_{i+1},u_i)$, $t^-_{\sg,\G}(u_i)=\sg(u_{i-1},u_i)$. Thus, \cref{s^+-s^-+t^+-t^-=0} yields 
	\begin{align}\label{sg(u_(i+1)_u_i)=sg(u_(i-1)_u_i)}
		\sg(u_{i+1},u_i)=\sg(u_{i-1},u_i).
	\end{align}

\textit{Case 3.} $u_{i-1}>u_i<u_{i+1}$. Then $s^+_{\sg,\G}(u_i)=\sg(u_i,u_{i+1})$, $s^-_{\sg,\G}(u_i)=\sg(u_i,u_{i-1})$, $t^\pm_{\sg,\G}(u_i)=0$. Thus, \cref{s^+-s^-+t^+-t^-=0} yields 
\begin{align}\label{sg(u_i_u_(i+1))=sg(u_i_u_(i-1))}
	\sg(u_i,u_{i+1})=\sg(u_i,u_{i-1}).
\end{align}

\textit{Case 4.} $u_{i-1}>u_i>u_{i+1}$. Then $s^+_{\sg,\G}(u_i)=t^-_{\sg,\G}(u_i)=0$, $s^-_{\sg,\G}(u_i)=\sg(u_i,u_{i-1})$, $t^+_{\sg,\G}(u_i)=\sg(u_{i+1},u_i)$. Thus, \cref{s^+-s^-+t^+-t^-=0} yields 
\begin{align}\label{sg(u_(i+1)_u_i)=sg(u_i_u_(i-1))}
	\sg(u_{i+1},u_i)=\sg(u_i,u_{i-1}).
\end{align}

Since $1\le i\le m-1$ was arbitrary, \cref{sg(u_i_u_(i+1))=sg(u_(i-1)_u_i),sg(u_(i+1)_u_i)=sg(u_(i-1)_u_i),sg(u_i_u_(i+1))=sg(u_i_u_(i-1)),sg(u_(i+1)_u_i)=sg(u_i_u_(i-1))} show that $\sg$ is constant on $\G$.

Conversely, assume that $\sg$ is constant on a cycle $\G:u_0,u_1,\dots,u_m=u_0$ and $x\in X$. If $x=u_i$ for some $1\le i\le m-1$, then \cref{s^+-s^-+t^+-t^-=0} at $x$ is equivalent to one of the equalities \cref{sg(u_i_u_(i+1))=sg(u_(i-1)_u_i),sg(u_(i+1)_u_i)=sg(u_(i-1)_u_i),sg(u_i_u_(i+1))=sg(u_i_u_(i-1)),sg(u_(i+1)_u_i)=sg(u_i_u_(i-1))} depending on the $4$ possible cases. If $x\not\in\G$, then \cref{s^+-s^-+t^+-t^-=0} at $x$ is trivial, because $s^\pm_{\sg,\G}(x)=t^\pm_{\sg,\G}(x)=0$.
\end{proof}

The map $\sg$ can take different values on distinct cycles of $X$, as the following example shows.
\begin{exm}
	Let $X=\{1,2,3,4,5,6,7\}$ with the following Hasse diagram.
	
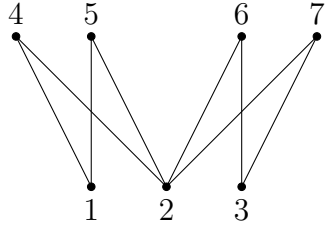
\begin{center}
		\begin{tikzpicture}
			\draw (0,0)-- (-1,2);
			\draw (-1,2)-- (-1,0);
			\draw (-1,0)-- (-2,2);
			\draw (-2,2)-- (0,0);
			\draw (0,0)-- (1,2);
			\draw (1,2)-- (1,0);
			\draw (1,0)-- (2,2);
			\draw (2,2)-- (0,0);
			\draw[fill=black] (0,0) circle (0.05);
			\draw[fill=black] (-1,2) circle (0.05);
			\draw[fill=black] (-1,0) circle (0.05);
			\draw[fill=black] (-2,2) circle (0.05);
			\draw[fill=black] (1,2) circle (0.05);
			\draw[fill=black] (2,2) circle (0.05);
			\draw[fill=black] (1,0) circle (0.05);
			\draw (-1,-0.3) node {$1$};
			\draw (0,-0.3) node {$2$};
			\draw (1,-0.3) node {$3$};
			\draw (-2,2.3) node {$4$};
			\draw (-1,2.3) node {$5$};
			\draw (1,2.3) node {$6$};
			\draw (2,2.3) node {$7$};
		\end{tikzpicture}
	\end{center}
Consider the $K$-linear map $\vf:I(X,K)\to I(X,K)$ defined as follows: 
$\vf(e_1)=e_1$,
$\vf(e_{14})=e_{14}$,
$\vf(e_{15})=e_{15}$,
$\vf(e_2)=e_2+e_3+e_6+e_7$,
$\vf(e_{24})=e_{24}$,
$\vf(e_{25})=e_{25}$,
$\vf(e_4)=e_4$,
$\vf(e_5)=e_5$,
$\vf(e_{26})=\vf(e_{27})=\vf(e_3)=\vf(e_{36})=\vf(e_{37})=\vf(e_6)=\vf(e_7)=0$.
Then $\vf\in\Dl(I(X,K))$ and $\vf$ satisfies \cref{vf(e_xy)=sg(x_y)e_xy} with $\sg(1,4)=\sg(1,5)=\sg(2,4)=\sg(2,5)=1$ and $\sg(2,6)=\sg(2,7)=\sg(3,6)=\sg(3,7)=0$. It is also diagonality preserving. 
%On $I(\{1,2,4,5\},K)$ the map $\vf$ acts as the identity plus a map with values in the centralizer of $I(\{1,2,4,5\},K)$, and $\vf(I(\{2,3,6,7\},K))\sst Z(I(\{2,3,6,7\},K))$. Finally, $\vf(I(\{1,4,5\},K))=I(\{1,4,5\},K)$, $\vf(I(\{3,6,7\},K))=\{0\}$ and
%\begin{align*}
%	[I(\{1,4,5\},K),I(\{2,3,6,7\},K)]=[I(\{1,2,4,5\},K),I(\{3,6,7\},K)]=\{0\}.
%\end{align*}
\end{exm}

\begin{lem}\label{existense-of-vf_sg_k}
	Let $u_0\in X$ and $\sg:X^2_<\to K$ a map constant on chains and cycles in $X$. Then there exists a unique diagonality preserving $\vf\in\Dl(I(X,K))$ such that $\vf(e_x)(u_0,u_0)=0$ for all $x\in X$ and $\sg_\vf=\sg$.	
\end{lem}
\begin{proof}
	Formula \cref{vf(e_xy)=sg(x_y)e_xy} is a unique way to define $\vf$ on $[I(X,K),I(X,K)]$. Then $\vf|_{[I(X,K),I(X,K)]}\in\Dl([I(X,K),I(X,K)])$ as observed in \cref{vf-on-[I(X_K)_I(X_K)]}.
	
	It remains to define $\vf|_{D(X,K)}$. Set $\vf(e_x)(u_0,u_0):=0$ for all $x\in X$. Given $v\in X$, choose a walk $\G: u_0,u_1,\dots,u_m=v$ from $u_0$ to $v$ and define $\vf(e_x)(v,v)$ by  \cref{vf(e_x)(u_m_u_m)=vf(e_x)(u_0_u_0)+s-and-t}. The definition does not depend on the choice of $\G$. Indeed, if there is another walk $\G': u'_0=u_0,u'_1,\dots,u'_{m'}=v$ from $u_0$ to $v$, then there is a closed walk $\Omega=\G*\G'\m$. Since $\sg$ is constant on cycles in $X$, it is admissible by \cref{sg-admissible<=>constant-on-cycles}, so in view of \cref{properties-of-s-and-t} we have
	\begin{align*}
		0&=s^+_{\sg,\Omega}(x)-s^-_{\sg,\Omega}(x)+t^+_{\sg,\Omega}(x)-t^-_{\sg,\Omega}(x)\\
		&=s^+_{\sg,\G}(x)+s^-_{\sg,\G'}(x)-s^-_{\sg,\G}(x)-s^+_{\sg,\G'}(x)+t^+_{\sg,\G}(x)+t^-_{\sg,\G'}(x)-t^-_{\sg,\G}(x)-t^+_{\sg,\G'}(x),
	\end{align*}
whence
\begin{align*}
	s^+_{\sg,\G}(x)-s^-_{\sg,\G}(x)+t^+_{\sg,\G}(x)-t^-_{\sg,\G}(x)=s^+_{\sg,\G'}(x)-s^-_{\sg,\G'}(x)+t^+_{\sg,\G'}(x)-t^-_{\sg,\G'}(x).
\end{align*}
	So, the right-hand side of \cref{vf(e_x)(u_m_u_m)=vf(e_x)(u_0_u_0)+s-and-t} does not depend on the walk from $u_0$ to $v$.
	
	Let us show that $\vf$ satisfies \cref{vf(e_x)(u_u)-vf(e_x)(v_v)-in-terms-of-sg} for all $u<v$. Indeed, there is a walk $\G: u_0,\dots,u_l=u,\dots,u_m=v$, where $u_l<\dots<u_m$. Denote $\G':u_0,\dots,u_l$ and $\G'':u_l,\dots,u_m$, so $\G=\G'*\G''$. Using $\G'$ and $\G$ to define $\vf(e_x)(u,u)$ and $\vf(e_x)(v,v)$, respectively, and applying \cref{properties-of-s-and-t}, we have
	\begin{align}
		\vf(e_x)(u,u)-\vf(e_x)(v,v)&=-s^+_{\sg,\G'}(x)+s^-_{\sg,\G'}(x)-t^+_{\sg,\G'}(x)+t^-_{\sg,\G'}(x)\notag\\
		&\quad+s^+_{\sg,\G}(x)-s^-_{\sg,\G}(x)+t^+_{\sg,\G}(x)-t^-_{\sg,\G}(x)\notag\\
		&=s^+_{\sg,\G''}(x)-s^-_{\sg,\G''}(x)+t^+_{\sg,\G''}(x)-t^-_{\sg,\G''}(x).\label{vf(e_x)(u_u)-vf(e_x)(v_v)=s^+-s^_+t^+-t^-}
	\end{align}
The proof of \cref{vf(e_x)(u_u)-vf(e_x)(v_v)-in-terms-of-sg} splits into the following $4$ cases, where we use \cref{s^+-and-s^-,t^+-and-t^-} and the fact that $\sg$ is constant on chains.

	\textit{Case 1.} $x=u$. Then $x=u_l$, so $s^+_{\sg,\G''}(x)=\sg(x,u_{l+1})$ and $s^-_{\sg,\G''}(x)=t^\pm_{\sg,\G''}(x)=0$. Therefore, \cref{vf(e_x)(u_u)-vf(e_x)(v_v)=s^+-s^_+t^+-t^-} equals $\sg(x,u_{l+1})=\sg(x,v)$. 
	
	\textit{Case 2.} $x=v$. Then $x=u_m$, so $t^-_{\sg,\G''}(x)=\sg(u_{m-1},u_m)$ and $s^\pm_{\sg,\G''}(x)=t^+_{\sg,\G''}(x)=0$. Therefore, \cref{vf(e_x)(u_u)-vf(e_x)(v_v)=s^+-s^_+t^+-t^-} equals $-\sg(u_{m-1},u_m)=-\sg(u,x)$. 
	
	\textit{Case 3.} $x=u_k$ for some $l<k<m$. Then $s^+_{\sg,\G}(x)=\sg(x,u_{k+1})$, $t^-_{\sg,\G}(x)=\sg(u_{k-1},x)$ and $s^-_{\sg,\G''}(x)=t^+_{\sg,\G''}(x)=0$. Therefore, \cref{vf(e_x)(u_u)-vf(e_x)(v_v)=s^+-s^_+t^+-t^-} equals $\sg(x,u_{k+1})-\sg(u_{k-1},x)=0$.
	
	\textit{Case 4.} $x\ne u_k$ for all $l\le k\le m$. Then $s^\pm_{\sg,\G''}(x)=t^\pm_{\sg,\G''}(x)=0$, and \cref{vf(e_x)(u_u)-vf(e_x)(v_v)=s^+-s^_+t^+-t^-} equals $0$.

%Thus, \cref{vf(e_x)(u_u)-vf(e_x)(v_v)-in-terms-of-sg} holds.

	Finally, define $\vf(e_x)=\sum_{v\in X}\vf(e_x)(v,v)e_v$. Obviously, $\vf|_{D(I(X,K))}\in\Dl(D(I(X,K)))$. We are left to show that $\vf$ acts as a $\frac 12$-derivation on $[e_x,e_{uv}]$ for arbitrary $x$ and $u<v$. There are $3$ cases, each of which uses \cref{vf(e_xy)=sg(x_y)e_xy,vf(e_x)(u_u)-vf(e_x)(v_v)-in-terms-of-sg}.
	
	\textit{Case 1.} $x=u$. Then $\vf([e_x,e_{uv}])=\vf(e_{xv})=\sg(x,v)e_{xv}$ and
	\begin{align*}
		[\vf(e_x),e_{uv}]&=(\vf(e_x)(u,u)-\vf(e_x)(v,v))e_{uv}=\sg(x,v)e_{xv},\\
		[e_x,\vf(e_{uv})]&=[e_x,\sg(x,v)e_{xv}]=\sg(x,v)e_{xv}.
	\end{align*}

\textit{Case 2.} $x=v$. Then $\vf([e_x,e_{uv}])=-\vf(e_{ux})=-\sg(u,x)e_{ux}$ and 
\begin{align*}
	[\vf(e_x),e_{uv}]&=(\vf(e_x)(u,u)-\vf(e_x)(v,v))e_{uv}=-\sg(u,x)e_{ux},\\
	[e_x,\vf(e_{uv})]&=[e_x,\sg(u,x)e_{ux}]=-\sg(u,x)e_{ux}.
\end{align*}

\textit{Case 3.} $x\not\in\{u,v\}$. Then $\vf([e_x,e_{uv}])=\vf(0)=0$ and
\begin{align*}
	[\vf(e_x),e_{uv}]&=(\vf(e_x)(u,u)-\vf(e_x)(v,v))e_{uv}=0,\\
	[e_x,\vf(e_{uv})]&=[e_x,\sg(u,v)e_{uv}]=0.
\end{align*}
In any case, $2\vf([e_x,e_{uv}])=[\vf(e_x),e_{uv}]+[e_x,\vf(e_{uv})]$. Thus, $\vf\in\Dl(I(X,K))$. By construction $\vf$ is diagonality preserving and $\sg_\vf=\sg$. In view of \cref{uniqueness-of-vf_sg_k} the restriction $\vf|_{D(X,K)}$ is uniquely determined by the condition $\vf(e_x)(u_0,u_0)=0$ for all $x\in X$. So, $\vf$ is unique.
\end{proof}

%\begin{rem}
%	In view of \cref{uniqueness-of-vf_sg_k} the $\frac 12$-derivation $\vf$ from \cref{existense-of-vf_sg_k} is uniquely determined by $\{k_x\}_{x\in X}\sst K$ such that $\vf(e_x)(u_0,u_0)=k_x$ for all $x\in X$.
%\end{rem}

\begin{defn}\label{defn-of-tau_sg_0_c}
	Fix $u_0\in X$. Given a map $\sg:X^2_<\to K$ which is constant on chains and cycles in $X$, define $\vf_\sg$ to be the diagonality preserving $\frac 12$-derivation of $I(X,K)$ such that $\vf_\sg|_{[I(X,K),I(X,K)]}$ is given by \cref{vf(e_xy)=sg(x_y)e_xy} and $\vf_\sg|_{D(X,K)}$ is uniquely determined
	by 
	\begin{align}\label{vf_sg(e_x)(u_0_u_0)=0}
		\vf_\sg(e_x)(u_0,u_0)=0
	\end{align}
	as in \cref{existense-of-vf_sg_k}.
\end{defn}

\begin{defn}
Given $\kp:X\to K$, define the central-valued $\vf_\kp\in\Dl(I(X,K))$ by means of
\begin{align}\label{vf_kp(e_xy)=kp(x).dl}
	\vf_\kp(e_{xy})=
	\begin{cases}
		\kp(x)\dl, & x=y,\\
		0, & x\ne y.
	\end{cases}
\end{align}
\end{defn}

\begin{prop}\label{vf-decomp-as-tau_0_sg_c}
	Fix $u_0\in X$. Then each diagonality preserving $\vf\in\Dl(I(X,K))$ is uniquely represented in the form
	\begin{align}\label{vf=vf_sg+vf_kp}
		\vf=\vf_\sg+\vf_\kp,
	\end{align}
	where $\sg:X^2_<\to K$ is a map which is constant on chains and cycles in $X$ and $\kp:X\to K$.
\end{prop}
\begin{proof}
	Let $\sg=\sg_\vf$. Then $\sg$ is constant on chains and cycles in $X$ by \cref{vf(e_xy)(x_y)=const}\cref{vf(e_xy)(x_y)=vf(e_uv)(u_v)} and \cref{sg-admissible<=>constant-on-cycles}.
	Hence, $\vf_\sg$ is well-defined. For all $x<y$ we have
	$
	\vf(e_{xy})=\sg(x,y)e_{xy}=\vf_\sg(e_{xy})=(\vf_\sg+\vf_\kp)(e_{xy}),
	$
	so $\vf|_{[I(X,K),I(X,K)]}=(\vf_\sg+\vf_\kp)|_{[I(X,K),I(X,K)]}$. Moreover, setting $\kp(x)=\vf(e_x)(u_0,u_0)$ for all $x\in X$, we have $\vf(e_x)(u_0,u_0)=(\vf_\sg+\vf_\kp)(e_x)(u_0,u_0)$ by \cref{vf_sg(e_x)(u_0_u_0)=0,vf_kp(e_xy)=kp(x).dl}, so $\vf|_{D(X,K)}=(\vf_\sg+\vf_\kp)|_{D(X,K)}$ by \cref{uniqueness-of-vf_sg_k}. Thus, \cref{vf=vf_sg+vf_kp} holds. The condition \cref{vf_sg(e_x)(u_0_u_0)=0} guarantees the uniqueness of $\kp$ in \cref{vf=vf_sg+vf_kp}, while the uniqueness of $\sg$ in \cref{vf=vf_sg+vf_kp} is given by \cref{vf(e_xy)=sg(x_y)e_xy}.
\end{proof}

As a consequence of \cref{vf-decomp-as-tau_0_sg_c,vf=inner+diag-pres}, we get the following complete characterization of $\frac 12$-derivations of $I(X,K)$.
\begin{thrm}\label{descr-half-der-I(X_K)}
	Let $X$ be finite and connected and fix $u_0\in X$. Then each $\vf\in\Dl(I(X,K))$ is of the form
	\begin{align*}
		\vf=\ad_c+\vf_\sg+\vf_\kp
	\end{align*}
for a unique $c\in Z([I(X,K),I(X,K)])$, a unique map $\sg:X^2_<\to K$ which is constant on chains and cycles in $X$ and a unique $\kp:X\to K$.
\end{thrm}

\section{Transposed Poisson structures on the incidence algebra}
We first point out $3$ classes of transposed Poisson structures on $I(X,K)$ that will appear in the final description. It is obvious that any transposed Poisson structure of Poisson type is of the form
\begin{align}\label{e_x-cdot-e_y=mu_xy.dl}
	e_x\cdot e_y=\mu(x,y)\dl
\end{align}
for some $\mu:X^2\to K$ with $\mu(x,y)=\mu(y,x)$, where the associativity of \cref{e_x-cdot-e_y=mu_xy.dl} is equivalent to
\begin{align*}
	\mu(x,y)\sum_{v\in X}\mu(z,v)=\mu(y,z)\sum_{v\in X}\mu(x,v).
\end{align*}
Observe that we write only non-trivial products in \cref{e_x-cdot-e_y=mu_xy.dl}. 

Each $\nu\in Z([I(X,K),I(X,K)])$ defines the mutational structure whose non-trivial products are:
	\begin{align}\label{e_x-cdot-e_y=-nu_xy.e_xy}
		e_x\cdot e_y=e_y\cdot e_x=[[e_x,\nu],e_y]=	
		\begin{cases}
			\nu(x,y)e_{xy}, & \Min(X)\ni x<y\in\Max(X),\\
			-\sum_{x<v\in\Max(X)}\nu(x,v)e_{xv}, & x=y\in\Min(X),\\
			-\sum_{\Min(X)\ni u<x}\nu(u,x)e_{ux}, & x=y\in\Max(X).
		\end{cases}
	\end{align}

%\begin{lem}\label{nu-structure-is-TP}
%	Any $\nu$-structure $\cdot$ is a transposed Poisson structure on $I(X,K)$ orthogonal to any structure of Poisson type.
%\end{lem}
%\begin{proof}
%	It is obvious that $\cdot$ is associative because all the products of $3$ elements are zero. For any $y\in\Max(X)$ define $c_y=\sum\nu(x,y)e_{xy}$, where the summation is over $\Min(X)\ni x<y$. Then $c_y\in Z([I(X,K),I(X,K)])$ and $e_y\cdot e_y=[c_y,e_y]$. Moreover, for any $x\in \Min(X)$ we have $e_y\cdot e_x=[c_y,e_x]$. Since, clearly, $[c_y,e_{uv}]=0=e_y\cdot e_{uv}$ for all $u<v$ and $u=v\not\in\Min(X)\sqcup\Max(X)$, the multiplication by $e_y$ is $\ad_{c_y}\in \Dl(I(X,K))$. Similarly, for any $x\in\Min(X)$ defining $c_x=-\sum\nu(x,y)e_{xy}$, where the summation is over $x<y\in\Max(X)$, we see that the multiplication by $e_x$ is $\ad_{c_x}\in \Dl(I(X,K))$. Thus, $\cdot$ is a transposed Poisson structure by \cref{glavlem}.
%	
%	If $x\in \Min(X)$, then $e_x\cdot\dl=\sum_{u\in\Min(X)} e_x\cdot e_u+\sum_{y\in\Max(X)} e_x\cdot e_y=\sum_{x<v\in\Max(X)} \nu(x,v)e_{xv}-\sum_{x<y\in\Max(X)}\nu(x,y)e_{xy}=0$. Similarly, $e_y\cdot\dl=0$ for all $y\in\Max(X)$. Hence, $\cdot$ is orthogonal to any structure \cref{e_x-cdot-e_y=mu_xy.dl} of Poisson type.
%\end{proof}

\begin{rem}\label{nu_xy-0-or-1}
	Applying a suitable automorphism $\phi\in\Aut(I(X,K),[\cdot,\cdot])$, we can make $\nu(x,y)\in\{0,1\}$ for all $\Min(X)\ni x<y\in\Max(X)$ in \cref{e_x-cdot-e_y=-nu_xy.e_xy}.
	
	Indeed, for all $x\le y$ define 
	\begin{align*}
		\phi(e_{xy})=
		\begin{cases}
			\nu(x,y)\m e_{xy}, & \Min(X)\ni x<y\in\Max(X)\text{ and }\nu(x,y)\ne 0,\\
			e_{xy}, & \text{otherwise}.
		\end{cases}
	\end{align*}
It is clear that $\phi$ is bijective and $\phi([e_{xy},e_{uv}])=[\phi(e_{xy}),\phi(e_{uv})]$ for $x<y$ and $u<v$, as well as for $x=y\not\in\{u,v\}$. Since also $\phi([e_x,e_{xy}])=\phi(e_{xy})=[e_x,\phi(e_{xy})]=[\phi(e_x),\phi(e_{xy})]$ and $\phi([e_{xy},e_y])=\phi(e_{xy})=[\phi(e_{xy}),e_y]=[\phi(e_{xy}),\phi(e_y)]$, we conclude that $\phi\in\Aut(I(X,K),[\cdot,\cdot])$. Applying $\phi$, we get the mutational structure corresponding to $\phi(\nu)=\sum_{\nu(x,y)\ne 0}e_{xy}$. 
\end{rem}

The definition of the third structure requires some preparation.
\begin{defn}
	We say that a pair $(x,y)$ of elements of $X$ is \textit{extreme}, if $x<y$ is a maximal chain in $X$ and there is no cycle in $X$ containing $x$ and $y$. Denote $X^2_e=\{(x,y)\in X^2\mid (x,y)\text{ is extreme}\}$.
\end{defn}

\begin{rem}\label{cycle-containing-x-and-y}
	Let $x<y$ with $l(x,y)=1$. If there is a cycle in $X$ containing $x$ and $y$, then there is a cycle in $X$ having $(x,y)$ as an edge.
	
	Indeed, if, for instance, $\G:u_0,u_1,\dots,u_i=x,\dots,u_j=y,\dots,u_m=u_0$ is a cycle, then $\G':u_0,u_1,\dots,u_i=x,u_j=y,\dots,u_m=u_0$ is also a cycle, because otherwise $i=1$, $m=j+1$, $u_0=y$ or $i=0$, $m=j+2$, $u_{j+1}=x$, where both cases lead to a repeated vertex in $\G$ ($u_0=u_j$ and $u_0=u_{j+1}$, respectively), a contradiction. 
\end{rem}

\begin{lem}\label{cycle-edge-(x_y)}
	Let $l(x,y)=1$. If there is a closed walk having $(x,y)$ (resp. $(y,x)$) as an edge and not having $(y,x)$ (resp. $(x,y)$) as an edge, then there is a cycle having $(x,y)$ (resp. $(y,x)$) as an edge.
\end{lem}
\begin{proof}
	Let $\G:u_0,\dots,u_i=x,u_{i+1}=y,\dots,u_m=u_0$ be a closed walk such that $(u_j,u_{j+1})\ne(y,x)$ for all $0\le j\le m-1$. If $u_k=u_l$ for some $0\le k\le i$ and $i+1\le l<m$, then replace $\G$ by $\G':u_k,\dots,u_i=x,u_{i+1}=y,\dots,u_l$. Observe that $\G'$, being a subwalk of $\G$, does not have $(y,x)$ as an edge. Repeating the procedure, if necessary, we may assume that the original closed walk $\G:u_0,\dots,u_i=x,u_{i+1}=y,\dots,u_m=u_0$ satisfies $u_k\ne u_l$ for all $0\le k\le i$ and $i+1\le l<m$. Now, if $u_k=u_l$ for some $0\le k<l\le i$, then consider the closed walk $\G':u_0,\dots,u_k,u_{l+1},\dots,u_i=x,u_{i+1}=y,\dots,u_m=u_0$. The edge $(y,x)$ can appear in $\G'$ only when $(u_k,u_{l+1})=(y,x)$, but this means that $\G$ has $(u_l,u_{l+1})=(y,x)$, which is impossible by the assumption. Similarly, if $u_k=u_l$ for some $i+1\le k<l\le m$, we replace $\G$ by $\G':u_0,\dots,u_i=x,u_{i+1}=y,\dots,u_k,u_{l+1},\dots,u_m=u_0$ which has no edges $(y,x)$ because otherwise $\G$ would have $(u_l,u_{l+1})=(y,x)$, a contradiction. Thus, replacing $\G$ by $\G'$ sufficiently many times, we may additionally assume that $u_k\ne u_l$ for $0\le k<l\le i$ and $i+1\le k<l\le m$. Since $m>2$ (otherwise $\G:x<y>x$ has edge $(y,x)$) the resulting $\G$ is a cycle having $(x,y)$ as an edge. 
	
	The case where $\G$ has an edge $(y,x)$ and no edge $(x,y)$ is analogous.
\end{proof}

\begin{lem}\label{path-ending-at-(x_y)-existense}
	Let $l(x,y)=1$ and $u_0\in X$. Then there is a path in $X$ starting at $u_0$ and ending at $(x,y)$ or $(y,x)$. 
\end{lem}
\begin{proof}
	If $u_0\in\{x,y\}$, the result is trivial. Let $u_0\not\in\{x,y\}$ and choose a path $u_0,\dots,u_k=x$ from $u_0$ to $x$. If $u_i\ne y$ for all $i$, then $\G:u_0,\dots,u_k=x,u_{k+1}=y$ is a desired path. Otherwise, $u_i=y$ for a unique $0<i<k$, so $\G:u_0,\dots,u_i=y,u_{i+1}=x$ is a desired path (it is indeed a path, because $u_j\ne x$ for all $0<j<k$).
\end{proof}

\begin{lem}\label{path-ending-at-(x_y)-uniqueness}
	Let $l(x,y)=1$ and $u_0\in X$. If there are two paths in $X$ starting at $u_0$ and ending at $(x,y)$ and $(y,x)$, respectively, then there is a cycle in $X$ having $(x,y)$ as an edge. 
\end{lem}
\begin{proof}
	Let $\G:u_0,\dots,u_{m-1}=x,u_m=y$ and $\Omega:v_0=u_0,\dots,v_{k-1}=y,u_k=x$ be two paths. Consider the closed walk $u_0,\dots,u_{m-1}=x,u_m=y=v_{k-1},v_{k-2},\dots,v_0=u$. It has $(x,y)$ as an edge and does not have $(y,x)$ as an edge, because $u_i\ne y$ for all $0\le i\le m-1$  and $v_j\ne x$ for all $0\le j\le k-1$. It remains to apply \cref{cycle-edge-(x_y)}.
%	
%	If it is not a cycle, then $u_i=v_j$ for some $0<i<m-1$ and $0<j<k-1$. Hence, $u_i,\dots,u_{m-1}=x,u_m=y=v_{k-1},v_{k-2},\dots,v_j$ is a closed walk of smaller cardinality. Then the trivial induction argument proves the statement.
\end{proof}

In view of \cref{path-ending-at-(x_y)-existense,path-ending-at-(x_y)-uniqueness} the following definition makes sense.
\begin{defn}
	Let $(x,y)\in X^2_e$ and $u_0\in X$. We set $\sgn_{u_0}(x,y):=1$, if there exists a path starting at $u_0$ and ending at $(x,y)$. Otherwise there exists a path starting at $u_0$ and ending at $(y,x)$, in which case set $\sgn_{u_0}(x,y):=-1$.
\end{defn}

\begin{lem}\label{at-most-one-edge-(x_y)}
	Let $(x,y)\in X^2_e$, $u_0,v\in X$ and $\G$ a path from $u_0$ to $v$ having $(x,y)$ or $(y,x)$ as an edge. If $\sgn_{u_0}(x,y)=1$ (resp. $\sgn_{u_0}(x,y)=-1$), then $\G$ has a unique edge of the form $(x,y)$ (resp. $(y,x)$) and no edges of the form $(y,x)$ (resp. $(x,y)$).
\end{lem}
\begin{proof}
	Consider the case $\sgn_{u_0}(x,y)=1$. If $\G$ had $(y,x)$ as an edge, then there would be a subpath of $\G$ starting at $u$ and ending at $(y,x)$, whence $\sgn_{u_0}(x,y)=-1$, a contradiction. %Obviously, being a path, $\G$ cannot have more than one edge of the form $(x,y).$
\end{proof}

\begin{defn}
	Let $l(x,y)=1$ and $u,v\in X$. We say that $u$ and $v$ are \textit{on the same side} with respect to $(x,y)$, if there is a path from $u$ to $v$ that does not have $(x,y)$ and $(y,x)$ as edges. Otherwise $u$ and $v$ are said to be \textit{on the opposite sides} with respect to $(x,y)$.
\end{defn}

\begin{rem}
	If $(x,y)\in X^2_e$ and $u,v$ are on the same side with respect to $(x,y)$, then no path from $u$ to $v$ has $(x,y)$ or $(y,x)$ as its edge, since otherwise by \cref{cycle-edge-(x_y),at-most-one-edge-(x_y)} there would exist a cycle containing either $(x,y)$ (whenever $\sgn_u(x,y)=1$) or $(y,x)$ (whenever $\sgn_u(x,y)=-1$) as an edge.
\end{rem}

The following lemma is a stronger version of \cref{at-most-one-edge-(x_y)}.

\begin{lem}\label{at-most-one-edge-(a_x)}
	Let $x,u\in X$. For each $v\in X$ there exists at most one $a<x$ such that 
	\begin{enumerate}
		\item $(a,x)\in X^2_e$ and $\sgn_u(a,x)=1$,
		\item $u,v$ are on the opposite sides with respect to $(a,x)$.
	\end{enumerate}  
The same holds if we replace ``$a<x$'' by ``$a>x$'' and ``$(a,x)$'' by ``$(x,a)$'', or ``$\sgn_u(a,x)=1$'' by ``$\sgn_u(a,x)=-1$''.
\end{lem}
\begin{proof}
	Let $a',a''<x$ such that $(a',x),(a'',x)\in X^2_e$, $\sgn_u(a',x)=\sgn_u(a'',x)=1$ and $u,v$ are on the opposite sides with respect to $(a',x)$ and $(a'',x)$. By \cref{at-most-one-edge-(x_y)} there are two paths from $u$ to $v$, say, $\G'$ and $\G''$, containing $(a',x)$ and $(a'',x)$ as an edge, respectively. Then $\G''$ does not contain $(a',x)$, otherwise $x$ would appear more than once in $\G''$. So $\G'*\G''\m$ is a closed walk having $(a',x)$ as an edge and not having $(x,a')$ as an edge. Hence, by \cref{cycle-edge-(x_y)} there is a cycle containing $(a',x)$, a contradiction.
\end{proof}

Fix $u_0\in X$. For any $(x,y)\in X^2_e$ denote
\begin{align*}%\label{V_xy=defn}
	V_{xy}=\{v\in X\mid u_0\text{ and }v\text{ are on the opposite sides with respect to }(x,y)\}.
\end{align*}

\begin{lem}\label{z<w-and-V_xy}
	Let $(x,y)$ be extreme.
	\begin{enumerate}
		\item\label{z-in-V_xy-iff-w-in-V_xy} If $z<w$ and $(x,y)\ne (z,w)$, then $z\in V_{xy}\iff w\in V_{xy}$;
		\item\label{x-in-V_xy-iff-(x_y)<0} $x\in V_{xy}\iff\sgn_{u_0}(x,y)=-1$;
		\item\label{y-in-V_xy-iff-(x_y)>0} $y\in V_{xy}\iff\sgn_{u_0}(x,y)=1$.
	\end{enumerate}
\end{lem}
\begin{proof}
	\textit{\cref{z-in-V_xy-iff-w-in-V_xy}}. Assume that $z\in V_{xy}$. Let $\G:u_0,\dots,u_m=z$, where $(u_i,u_{i+1})\in\{(x,y),(y,x)\}$ for some (unique) $0\le i\le m-1$. Fix a chain $C:z=z_0<\dots<z_l=w$ with $l(z_j,z_{j+1})=1$ for all $0\le j\le l-1$. Observe that $(z_j,z_{j+1})\ne (x,y)$ for all $0\le j\le l-1$: whenever $l>1$, this is due to the maximality of the chain $x<y$, and whenever $l=1$, this is thanks to $(x,y)\ne (z,w)$. Obviously, $(z_j,z_{j+1})\ne (y,x)$ for all $0\le j\le l-1$, because $y>x$. Let $k=\min\{0\le s\le m\mid u_s\in C\}$ (which exists, because $u_m=z_0$) and $0\le j\le l$ such that $u_k=z_j$. If $0\le k\le i$, then $\G':u_0,\dots,u_k=z_j,z_{j-1},\dots z_0=z$ is a path not containing $(x,y)$ and $(y,x)$, a contradiction with $z\in V_{xy}$. If $i+1\le k\le m$, then $\G':u_0,\dots,u_k=z_j,z_{j+1},\dots,z_l=w$ is a path containing $(x,y)$ or $(y,x)$, so $w\in V_{xy}$. The implication $w\in V_{xy}\impl z\in V_{xy}$ is proved similarly.
	
	\textit{\cref{x-in-V_xy-iff-(x_y)<0}}. If $x\in V_{xy}$ and $\G:u_0,\dots,u_m=x$ a path with $(u_i,u_{i+1})\in\{(x,y),(y,x)\}$ for some (unique) $0\le i\le m-1$, then $i=m-1$ and $(u_i,u_{i+1})=(y,x)$, so $\sgn_{u_0}(x,y)=-1$. Conversely, if $\sgn_{u_0}(x,y)=-1$, then there is path starting at $u_0$ and ending at $(y,x)$, so $x\in V_{xy}$.
	
	\textit{\cref{y-in-V_xy-iff-(x_y)>0}}. The proof is similar to that of \cref{x-in-V_xy-iff-(x_y)<0}.
\end{proof}

\begin{defn}
	Let $\lb:X^2_e\to K$. Define the following multiplication on $I(X,K)$:
	\begin{align}
		e_x\cdot e_{xy}&=e_{xy}\cdot e_x=-e_{xy}\cdot e_y=-e_y\cdot e_{xy}=\lb(x,y)e_{xy},\text{ if }(x,y)\in X^2_e,\label{e_x-cdot-e_xy=-e_xy-cdot-y=lb_xy.e_xy}\\
		e_x\cdot e_y&=e_x\cdot e_y=
		\begin{cases}
			\sgn_{u_0}(x,y)\lb(x,y)e_{V_{xy}},& (x,y)\in X^2_e,\\
			-\sum_{(x,v)\in X^2_e}\sgn_{u_0}(x,v)\lb(x,v)e_{V_{xv}}, & x=y\in\Min(X),\\
			-\sum_{(u,x)\in X^2_e}\sgn_{u_0}(u,x)\lb(u,x)e_{V_{ux}}, & x=y\in\Max(X).
		\end{cases}\label{e_x-cdot-e_y=sums-with-lambdas}
	\end{align}
We call $\cdot$ a \textit{$\lb$-structure} on $I(X,K)$.
\end{defn}

\begin{rem}\label{e_x-cdot-dl=0}
	For any $\lb$-structure $\cdot$ and for all $x\in X$ we have $e_x\cdot\dl=\dl\cdot e_x=0$.
	
	Clearly, $\cdot$ is commutative, so it suffices to prove $e_x\cdot\dl=0$. If $x\not\in\Min(X)\sqcup \Max(X)$, then $e_x\cdot e_y=0$ for all $y\in X$, whence $e_x\cdot\dl=\sum e_x\cdot e_y=0$. If
	$x\in\Min(X)$, then $e_x\cdot\dl=e_x\cdot e_x+\sum_{(x,y)\in X^2_e} e_x\cdot e_y=0$ by \cref{e_x-cdot-e_y=sums-with-lambdas}. The case $x\in\Max(X)$ is analogous.
\end{rem}

\begin{lem}\label{lb-products-with-e_V_xy}
	Let $(x,y)\in X^2_e$ and $\cdot$ a $\lb$-structure.
	\begin{enumerate}
		\item\label{e_V_xy-cdot-e_zw=0} If $z<w$ and $(x,y)\ne (z,w)$, then $e_{V_{xy}}\cdot e_{zw}=0$.
		\item\label{e_V_xy-cdot-e_xy} $e_{V_{xy}}\cdot e_{xy}=-\sgn_{u_0}(x,y)\lb(x,y)e_{xy}$.
		\item\label{e_V_xy-cdot_lb-e_z=0} If $z\not\in\{x,y\}$, then $e_{V_{xy}}\cdot e_z=0$.
		\item\label{e_V_xy-cdot_lb-e_x-or-e_y}	$e_{V_{xy}}\cdot e_x=\sgn_{u_0}(x,y)\lb(x,y)e_{V_{xy}}=-e_{V_{xy}}\cdot e_y$. 
	\end{enumerate}
\end{lem}
\begin{proof}
	\textit{\cref{e_V_xy-cdot-e_zw=0}}. If $z<w$ and $(x,y)\ne (z,w)$, then either $z,w\not\in V_{xy}$ or $z,w\in V_{xy}$ by \cref{z<w-and-V_xy}\cref{z-in-V_xy-iff-w-in-V_xy}. In the former case $e_v\cdot e_{zw}=0$ for all $v\in V_{xy}$ by \cref{e_x-cdot-e_xy=-e_xy-cdot-y=lb_xy.e_xy}, so $e_{V_{xy}}\cdot e_{zw}=\sum_{v\in V_{xy}}e_v\cdot e_{zw}=0$. In the latter case $e_{V_{xy}}\cdot e_{zw}=(e_z+e_w)\cdot e_{zw}=0$.
	
	\textit{\cref{e_V_xy-cdot-e_xy}}. If $\sgn_{u_0}(x,y)=1$, then $x\not\in V_{xy}$ and $y\in V_{xy}$ by \cref{x-in-V_xy-iff-(x_y)<0,y-in-V_xy-iff-(x_y)>0} of \cref{z<w-and-V_xy}, whence $e_{V_{xy}}\cdot e_{xy}=\sum_{v\in V_{xy}}e_v\cdot e_{xy}=e_y\cdot e_{xy}=-\lb(x,y)e_{xy}$ by \cref{e_x-cdot-e_xy=-e_xy-cdot-y=lb_xy.e_xy}. Otherwise $x\in V_{xy}$ and $y\not\in V_{xy}$, so $e_{V_{xy}}\cdot e_{xy}=e_x\cdot e_{xy}=\lb(x,y)e_{xy}$.
	
	\textit{\cref{e_V_xy-cdot_lb-e_z=0}}. If $z\not\in\Min(X)\sqcup \Max(X)$, the result is trivial. Let $z\in\Min(X)$. For any  $z<v$ we have $(z,v)\ne (x,y)$ because $z\ne x$. Therefore, either $z,v\not\in V_{xy}$ or $z,v\in V_{xy}$ by \cref{z<w-and-V_xy}\cref{z-in-V_xy-iff-w-in-V_xy}. In the former case $e_z\cdot e_v=0$ for all $v\in V_{xy}$, whence $e_z\cdot e_{V_{xy}}=\sum_{v\in V_{xy}}e_z\cdot e_v=0$. In the latter case $e_z\cdot e_{V_{xy}}=e_z\cdot e_z+\sum_{z\ne v} e_z\cdot e_v=e_z\cdot\dl=0$ by \cref{e_x-cdot-dl=0}. The case $z\in\Max(X)$ is similar.
	
	\textit{\cref{e_V_xy-cdot_lb-e_x-or-e_y}}. If $\sgn_{u_0}(x,y)=1$, then $x\not\in V_{xy}$ and $y\in V_{xy}$ by \cref{x-in-V_xy-iff-(x_y)<0,y-in-V_xy-iff-(x_y)>0} of \cref{z<w-and-V_xy}. For any $x<v\ne y$ we have $v\not\in V_{xy}$ by \cref{z<w-and-V_xy}\cref{z-in-V_xy-iff-w-in-V_xy}. Hence, $e_x\cdot e_{V_{xy}}=\sum_{v\in V_{xy}}e_x\cdot e_v=e_x\cdot e_y=\lb(x,y)e_{V_{xy}}$. Since $y\in V_{xy}$, then for any $x\ne z<y$ we have $z\in V_{xy}$ due to \cref{z<w-and-V_xy}\cref{z-in-V_xy-iff-w-in-V_xy}. Consequently, $e_{V_{xy}}\cdot e_y=e_y\cdot e_y+\sum_{z\not\in \{x,y\}} e_z\cdot e_y=(\dl-e_x)\cdot e_y=-e_x\cdot e_y=-\lb(x,y)e_{V_{xy}}$ thanks to \cref{e_x-cdot-dl=0}. The case $\sgn_{u_0}(x,y)=-1$ is similar.
\end{proof}

\begin{lem}\label{lb-structure-assoc}
	Any $\lb$-structure $\cdot$ is associative.
\end{lem}
\begin{proof}
	Clearly, it suffices to consider only the products $(f\cdot g)\cdot h$ and $f\cdot (g\cdot h)$ of basis elements involving at most one $e_{xy}$ with $x<y$, in which case $(x,y)$ can (and will) be assumed to be extreme.
	
	\textit{Case 1.} $(f,g,h)=(e_u,e_{xy},e_v)$. If $\{u,v\}\not\sst\{x,y\}$, then $(f\cdot g)\cdot h=0=f\cdot(g\cdot h)$. Moreover, $(f\cdot g)\cdot h=-\lb(x,y)^2e_{xy}=f\cdot(g\cdot h)$ for $\{u,v\}=\{x,y\}$. The case $u=v\in\{x,y\}$ is trivial.
	
	\textit{Case 2.} $(f,g,h)=(e_u,e_v,e_{xy})$ with $u\ne v$. If $(u,v)\not\in X^2_e$ or $(v,u)\not\in X^2_e$, then $(f\cdot g)\cdot h=0=f\cdot (g\cdot h)$ by \cref{e_x-cdot-e_xy=-e_xy-cdot-y=lb_xy.e_xy,e_x-cdot-e_y=sums-with-lambdas}. Let $(u,v)\in X^2_e$. Then $(f\cdot g)\cdot h=\sgn_{u_0}(u,v)\lb(u,v)e_{V_{uv}}\cdot e_{xy}$. If $(u,v)\ne (x,y)$, then $e_{V_{uv}}\cdot e_{xy}=0$ by \cref{lb-products-with-e_V_xy}\cref{e_V_xy-cdot-e_zw=0}. Hence, $(f\cdot g)\cdot h=0$, which obviously coincides with $f\cdot (g\cdot h)$. For $(u,v)=(x,y)$ we have $(f\cdot g)\cdot h=\sgn_{u_0}(x,y)\lb(x,y)e_{V_{xy}}\cdot e_{xy}=-\lb(x,y)^2 e_{xy}$ by \cref{lb-products-with-e_V_xy}\cref{e_V_xy-cdot-e_xy}. Clearly, $f\cdot (g\cdot h)=-\lb(x,y)e_x\cdot e_{xy}=-\lb(x,y)^2 e_{xy}$ by \cref{e_x-cdot-e_xy=-e_xy-cdot-y=lb_xy.e_xy}. The case $(v,u)\in X^2_e$ is symmetric.
	
	\textit{Case 3.} $(f,g,h)=(e_u,e_u,e_{xy})$. The case $u\not\in\Min(X)\sqcup \Max(X)$ is trivial. Let $u\in\Min(X)$. If $u\ne x$, then $(x,y)\ne(u,v)$ for any $(u,v)\in X^2_e$, so $e_{V_{uv}}\cdot e_{xy}=0$ by \cref{lb-products-with-e_V_xy}\cref{e_V_xy-cdot-e_zw=0}, whence $(f\cdot g)\cdot h=0$ by \cref{e_x-cdot-e_y=sums-with-lambdas}, which coincides with $f\cdot (g\cdot h)$. As to the case $u=x$, we have $e_{V_{xv}}\cdot e_{xy}=0$ for all $(x,v)\in X^2_e\setminus\{(x,y)\}$ by \cref{lb-products-with-e_V_xy}\cref{e_V_xy-cdot-e_zw=0}. Therefore, $(f\cdot g)\cdot h=-\sgn_{u_0}(x,y)\lb(x,y)e_{V_{xy}}\cdot e_{xy}=\lb(x,y)^2e_{xy}$ by \cref{lb-products-with-e_V_xy}\cref{e_V_xy-cdot-e_xy} and \cref{e_x-cdot-e_y=sums-with-lambdas}. This coincides with $f\cdot (g\cdot h)$. The case $u\in\Max(X)$ is similar.
	
	\textit{Case 4.} $(f,g,h)=(e_x,e_y,e_z)$ with $x,y,z$ pairwise distinct. Assume that $f\cdot g\ne 0$. It follows that $(x,y)\in X^2_e$ or $(y,x)\in X^2_e$. Let $(x,y)\in X^2_e$ (otherwise commute $e_x$ and $e_y$). Then $(f\cdot g)\cdot h=\sgn_{u_0}(x,y)\lb(x,y) e_{V_{xy}}\cdot e_z=0$ by \cref{lb-products-with-e_V_xy}\cref{e_V_xy-cdot_lb-e_z=0}. By the same argument $f\cdot (g\cdot h)=(g\cdot h)\cdot f=0$.
	
	\textit{Case 5.} $(f,g,h)=(e_x,e_x,e_y)$ with $x\ne y$. Assume that $(x,y)\in X^2_e$, so we have $f\cdot (g\cdot h)=\sgn_{u_0}(x,y)\lb(x,y)e_x\cdot e_{V_{xy}}=\lb(x,y)^2e_{V_{xy}}$ by \cref{lb-products-with-e_V_xy}\cref{e_V_xy-cdot_lb-e_x-or-e_y}. As to $(f\cdot g)\cdot h$, for any extreme $(x,u)\in X^2_e\setminus\{(x,y)\}$ we have $e_{V_{xu}}\cdot e_y=0$ by \cref{lb-products-with-e_V_xy}\cref{e_V_xy-cdot_lb-e_z=0}. So $(f\cdot g)\cdot h=-\sgn_{u_0}(x,y)\lb(x,y)e_{V_{xy}}\cdot e_y=\lb(x,y)^2e_{V_{xy}}$ by \cref{e_x-cdot-e_y=sums-with-lambdas} and \cref{lb-products-with-e_V_xy}\cref{e_V_xy-cdot_lb-e_x-or-e_y}. The case $(y,x)\in X^2_e$ similarly gives $(f\cdot g)\cdot h=-\lb(y,x)^2e_{V_{yx}}=f\cdot (g\cdot h)$. It remains to consider the situation $(x,y),(y,x)\not\in X^2_e$, so that $f\cdot (g\cdot h)=0$. The non-trivial case is $x\in\Min(X)\sqcup\Max(X)$. If $x\in\Min(X)$, then $y\ne u$ for all $(x,u)\in X^2_e$, since otherwise $(x,y)=(x,u)\in X^2_e$. Thanks to \cref{lb-products-with-e_V_xy}\cref{e_V_xy-cdot_lb-e_z=0}, this implies that $e_{V_{xu}}\cdot e_y=0$, whence $(f\cdot g)\cdot h=0$ by \cref{e_x-cdot-e_y=sums-with-lambdas}. The case $x\in\Max(X)$ is analogous.
\end{proof}

\begin{lem}\label{lb-structure-is-TP}
		Any $\lb$-structure $\cdot$ is a transposed Poisson structure on $I(X,K)$ orthogonal to any structure of Poisson type.
\end{lem}
\begin{proof}
	The product $\cdot$ is commutative by definition and associative by \cref{lb-structure-assoc}. It is also orthogonal to any product of the form \cref{e_x-cdot-e_y=mu_xy.dl} due to \cref{e_x-cdot-dl=0}. In view of \cref{glavlem}, it remains to prove that $e_{xy}\cdot-$ is a $\frac 12$-derivation of $I(X,K)$ for all $x\le y$.
	
	For any $x\in X$ define $\sg_x:X^2_<\to K$ as follows:
	\begin{align}\label{sg_x(u_v)=lb_xv-or-lb_ux}
		\sg_x(u,v)=
		\begin{cases}
			\lb(x,v), & x=u\text{ and }(x,v)\in X^2_e,\\
			-\lb(u,x), & x=v\text{ and }(u,x)\in X^2_e,\\
			0, & \text{otherwise}.
		\end{cases}
	\end{align}
Observe that $\sg_x$ is constant on chains in $X$ (in fact, it is zero on any chain of length $>1$). Since no cycle in $X$ contains an extreme pair as an edge, $\sg_x$ is constant (in fact, zero) on cycles in $X$. 

Consider $\vf_{\sg_x}$ as in \cref{defn-of-tau_sg_0_c}. Then, clearly, for all $u<v$
\begin{align*}
	\vf_{\sg_x}(e_{uv})=\sg_x(u,v)e_{uv}&=
	\begin{cases}
		\lb(x,v)e_{xv}, & x=u\text{ and }(x,v)\in X^2_e,\\
		-\lb(u,x)e_{ux}, & x=v\text{ and }(u,x)\in X^2_e,\\
		0, & \text{otherwise}.
	\end{cases}
\\
&=e_x\cdot e_{uv}.
\end{align*}

 Take $x,y$ in $X$. If $x\not\in\Min(X)\sqcup\Max(X)$ or $y\not\in\Min(X)\sqcup\Max(X)$, then $\sg_x(u,y)=\sg_x(y,v)=0$ for all $u<y<v$ by \cref{sg_x(u_v)=lb_xv-or-lb_ux}. Consequently, $s^\pm_{\sg_x,\G}(y)=t^\pm_{\sg_x,\G}(y)=0$ for any walk $\G$. Hence, by \cref{vf(e_x)(u_m_u_m)=vf(e_x)(u_0_u_0)+s-and-t}
 \begin{align*}%\label{e_x-cdot-e_y=0-x-or-y-not-min-or-max}
	\vf_{\sg_x}(e_y)=0=e_x\cdot e_y,\text{ if }x\not\in\Min(X)\sqcup\Max(X)\text{ or }y\not\in \Min(X)\sqcup\Max(X).
\end{align*}

Let $x\in\Min(X)$ and $y\in\Max(X)$. Given $v\in X$, consider a walk $\G:u_0,\dots,u_m=v$ from $u_0$ to $v$. Without loss of generality assume that $\G$ has at most one edge of the form $(x,y)$ or $(y,x)$ (otherwise remove a closed subwalk from $\G$). We have $s^\pm_{\sg_x,\G}(y)=0$, because $y\in\Max(X)$. Furthermore, by \cref{sg_x(u_v)=lb_xv-or-lb_ux}
\begin{align*}
	t^+_{\sg_x,\G}(y)&=\sum_{y=u_i>u_{i+1}} \sg_x(u_{i+1},y)=\sum_{y=u_i>u_{i+1}=x} \sg_x(x,y), \\ 
	t^-_{\sg_x,\G}(y)&=\sum_{u_i<u_{i+1}=y} \sg_x(u_i,y)=\sum_{x=u_i<u_{i+1}=y} \sg_x(x,y),
\end{align*}
so by \cref{vf(e_x)(u_m_u_m)=vf(e_x)(u_0_u_0)+s-and-t,sg_x(u_v)=lb_xv-or-lb_ux}
\begin{align*}
	\vf_{\sg_x}(e_y)(v,v)=
	\begin{cases}
		\lb(x,y), & (x,y)\in X^2_e\text{ and }\exists i: (u_i,u_{i+1})=(x,y),\\
		-\lb(x,y), & (x,y)\in X^2_e\text{ and }\exists i: (u_i,u_{i+1})=(y,x),\\
		0, & \text{otherwise}.
	\end{cases}
\end{align*}
It follows that
\begin{align*}
	\vf_{\sg_x}(e_y)&=
	\begin{cases}
		\sgn_{u_0}(x,y)\lb(x,y)e_{V_{xy}},& (x,y)\in X^2_e,\\
		0, & \text{otherwise}.
	\end{cases}\\
&=e_x\cdot e_y.
\end{align*}

Let $x\in\Max(X)$ and $y\in\Min(X)$. Considering a walk $\G:u_0,\dots,u_m=v$ from $u_0$ to $v$ as above, we have $t^\pm_{\sg_x,\G}(y)=0$, while
\begin{align*}
	s^+_{\sg_x,\G}(y)&=\sum_{y=u_i<u_{i+1}} \sg_x(y,u_{i+1})=\sum_{y=u_i<u_{i+1}=x} \sg_x(y,x), \\ 
	s^-_{\sg_x,\G}(y)&=\sum_{u_i>u_{i+1}=y} \sg_x(y,u_i)=\sum_{x=u_i>u_{i+1}=y} \sg_x(y,x),
\end{align*}
so by \cref{vf(e_x)(u_m_u_m)=vf(e_x)(u_0_u_0)+s-and-t,sg_x(u_v)=lb_xv-or-lb_ux}
\begin{align*}
	\vf_{\sg_x}(e_y)(v,v)=
	\begin{cases}
		\lb(y,x), & (y,x)\in X^2_e\text{ and }\exists i: (u_i,u_{i+1})=(y,x),\\
		-\lb(y,x), & (y,x)\in X^2_e\text{ and }\exists i: (u_i,u_{i+1})=(x,y),\\
		0, & \text{otherwise}.
	\end{cases}
\end{align*}
It follows that
\begin{align*}
	\vf_{\sg_x}(e_y)&=
	\begin{cases}
		\sgn_{u_0}(y,x)\lb(y,x)e_{V_{yx}},& (y,x)\in X^2_e,\\
		0, & \text{otherwise}.
	\end{cases}\\
	&=e_x\cdot e_y.
\end{align*}

Now consider the case $x,y\in\Max(X)$. For any walk $\G$ from $u_0$ to $v$ we have $s^\pm_{\sg_x,\G}(y)=0$, because $y\in\Max(X)$. If $x\ne y$, then
\begin{align*}
	t^+_{\sg_x,\G}(y)=\sum_{y=u_i>u_{i+1}} \sg_x(u_{i+1},y)=0\text{ and }t^-_{\sg_x,\G}(y)=\sum_{u_i<u_{i+1}=y} \sg_x(u_i,y)=0
\end{align*}
by \cref{sg_x(u_v)=lb_xv-or-lb_ux}, because $x$ is not comparable with $y$. Consequently,
\begin{align*}%\label{e_x-cdot-e_y=0-x-max-y-max-x-ne-y}
	\vf_{\sg_x}(e_y)=0=e_x\cdot e_y,\text{ if }x,y\in \Max(X),\ x\ne y.
\end{align*}
For the case $x=y$ write
\begin{align*}
	t^+_{\sg_x,\G}(x)=\sum_{x=u_i>u_{i+1}} \sg_x(u_{i+1},x)\text{ and }t^-_{\sg_x,\G}(x)=\sum_{u_i<u_{i+1}=x} \sg_x(u_i,x).
\end{align*}
We may assume that $x$ appears among the vertices of $\G$ at most once (otherwise remove the corresponding closed subwalk of $\G$). If $x=u_0=x_1$, then $t^+_{\sg_x,\G}(x)=\sg_x(u_1,x)=-\lb(u_1,x)$ and $t^-_{\sg_x,\G}(x)=0$. If $x=u_m=v$, then $t^+_{\sg_x,\G}(x)=0$ and $t^-_{\sg_x,\G}(x)=\sg_x(u_{m-1},x)=-\lb(u_{m-1},x)$. If $x=u_i$ with $0<i<m$, then $t^+_{\sg_x,\G}(x)=\sg_x(u_{i+1},x)=-\lb(u_{i+1},x)$ and $t^-_{\sg_x,\G}(x)=\sg_x(u_{i-1},x)=-\lb(u_{i-1},x)$. Thus,
\begin{align}\label{vf_sg_x(e_x)(v_v)}
	\vf_{\sg_x}(e_x)(v,v)=
	\begin{cases}
		\lb(u_1,x), & x=u_0=x_1,\\
		-\lb(u_{m-1},x), & x=u_m=v,\\
		\lb(u_{i+1},x)-\lb(u_{i-1},x), & x=u_i,\ 0<i<m,\\
		0, & \text{otherwise}.
	\end{cases}
\end{align}
Clearly, it suffices to consider only extreme pairs in \cref{vf_sg_x(e_x)(v_v)}. Observe that the sign of $\lb(u_k,x)$ in \cref{vf_sg_x(e_x)(v_v)} coincides with $-\sgn_{u_0}(u_k,x)$. By \cref{at-most-one-edge-(a_x)} each $v\in X$ belongs to at most one $V_{zx}$, where $(z,x)\in X^2_e$ and $\sgn_{u_0}(z,x)=1$, and to at most one $V_{wx}$, where $(w,x)\in X^2_e$ and $\sgn_{u_0}(w,x)=-1$. Consequently,
\begin{align*}%\label{e_x-cdot-e_x-for-x-max}
	\vf_{\sg_x}(e_x)=-\sum\sgn_{u_0}(u,x)\lb(u,x)e_{V_{ux}}=e_x\cdot e_x.
\end{align*}

Finally, let $x,y\in\Min(X)$. Then $t^\pm_{\sg_x,\G}(y)=0$. If $x\ne y$, then
\begin{align*}
	s^+_{\sg_x,\G}(y)=\sum_{y=u_i<u_{i+1}} \sg_x(y,u_{i+1})\text{ and }s^-_{\sg_x,\G}(y)=\sum_{u_i>u_{i+1}=y} \sg_x(y,u_i)=0,
\end{align*}
so
\begin{align*}%\label{e_x-cdot-e_y=0-x-max-y-min-x-ne-y}
	\vf_{\sg_x}(e_y)=0=e_x\cdot e_y,\text{ if }x,y\in \Min(X),\ x\ne y.
\end{align*}
Furthermore,
\begin{align*}
	s^+_{\sg_x,\G}(x)=\sum_{x=u_i<u_{i+1}} \sg_x(x,u_{i+1})\text{ and }s^-_{\sg_x,\G}(x)=\sum_{u_i>u_{i+1}=x} \sg_x(x,u_i).
\end{align*}
It follows that
\begin{align*}
	\vf_{\sg_x}(e_x)(v,v)=
	\begin{cases}
		-\lb(x,u_1), & x=u_0=x_1,\\
		\lb(x,u_{m-1}), & x=u_m=v,\\
		-\lb(x,u_{i+1})+\lb(x,u_{i-1}), & x=u_i,\ 0<i<m,\\
		0, & \text{otherwise}.
	\end{cases}
\end{align*}
Consequently,
\begin{align*}%\label{e_x-cdot-e_x-for-x-min}
	\vf_{\sg_x}(e_x)=-\sum\sgn_{u_0}(x,v)\lb(x,v)e_{V_{xv}}=e_x\cdot e_x.
\end{align*}

Thus, we have proved that $e_x\cdot -$ coincides with $\vf_{\sg_x}\in\Dl(I(X,K))$. Clearly, by \cref{e_x-cdot-e_xy=-e_xy-cdot-y=lb_xy.e_xy}, for any $(x,y)\in X^2_e$ the map $e_{xy}\cdot -$ is $\ad_{c_{xy}}\in\Dl(I(X,K))$, where $c_{xy}=-\lb(x,y)e_{xy}$. For $(x,y)\not\in X^2_e$ the map $e_{xy}\cdot -$ is zero.
\end{proof}

We need some auxiliary equalities analogous to that of \cref{lb-products-with-e_V_xy}, but for a mutational structure.
\begin{lem}\label{mu-products-with-e_V_xy}
	Let $(x,y)\in X^2_e$ and $\cdot$ a mutational structure.
	\begin{enumerate}
		\item\label{e_V_xy-cdot_mu-e_z=0} If $z\not\in\{x,y\}$, then $e_{V_{xy}}\cdot e_z=0$.
		\item\label{e_V_xy-cdot_mu-e_x-or-e_y}	$e_{V_{xy}}\cdot e_x=\sgn_{u_0}(x,y)\nu(x,y)e_{xy}=-e_{V_{xy}}\cdot e_y$. 
	\end{enumerate}
\end{lem}
\begin{proof}
	\textit{\cref{e_V_xy-cdot_mu-e_z=0}}. The proof is the same as that of \cref{lb-products-with-e_V_xy}\cref{e_V_xy-cdot_lb-e_z=0}, but we do not need \cref{e_x-cdot-dl=0} to justify $e_z\cdot\dl=0$, because the latter holds by the orthogonality of $\cdot$ to a structure of Poisson type.
	
	\textit{\cref{e_V_xy-cdot_mu-e_x-or-e_y}}. As in the proof of \cref{lb-products-with-e_V_xy}\cref{e_V_xy-cdot_lb-e_x-or-e_y}, whenever $\sgn_{u_0}(x,y)=1$, we have $e_{V_{xy}}\cdot e_x=e_x\cdot e_y=\nu(x,y)e_{xy}$ and $e_{V_{xy}}\cdot e_y=(\dl-e_x)\cdot e_y=-e_x\cdot e_y=-\nu(x,y)e_{xy}$. The case $\sgn_{u_0}(x,y)=-1$ is similar.
\end{proof}

\begin{lem}\label{lb+nu-structure-is-TP}
	The sum of any mutational structure and any $\lb$-structure is a transposed Poisson structure on $I(X,K)$.
\end{lem}
\begin{proof}
	To distinguish the two structures, denote a mutational structure by $\cdot_\nu$ and a $\lb$-structure by $\cdot_\lb$. We only need to prove that their sum is associative. Since both of the structures are themselves associative, it suffices to establish the equalities of the mixed products of basis elements: 
	\begin{align}\label{(f-cdot_lb-g)-cdot_mu-h=f-cdot_lb-(g-cdot_mu-h)}
		(f\cdot_\lb g)\cdot_\mu h=f\cdot_\lb (g\cdot_\mu h)\text{ and }(f\cdot_\mu g)\cdot_\lb h=f\cdot_\mu (g\cdot_\lb h).
	\end{align} If $e_{xy}\in\{f,g,h\}$ with $x<y$, then all the products in \cref{(f-cdot_lb-g)-cdot_mu-h=f-cdot_lb-(g-cdot_mu-h)} are zero.
	
	\textit{Case 1.} $(f,g,h)=(e_x,e_y,e_z)$ with $x,y,z$ pairwise distinct. Assume that $f\cdot_\lb g\ne 0$. Then either $(x,y)\in X^2_e$ or $(y,x)\in X^2_e$. If $(x,y)\in X^2_e$, then $(f\cdot_\lb g)\cdot_\mu h=\sgn_{u_0}(x,y)\lb(x,y)e_{V_{xy}}\cdot_\mu e_z=0$ by \cref{mu-products-with-e_V_xy}\cref{e_V_xy-cdot_mu-e_z=0}. The case $(y,x)\in X^2_e$ similarly yields $(f\cdot_\lb g)\cdot_\mu h=0$. Assume that $g\cdot_\mu h\ne 0$ and let $\Min(X)\ni y<z\in\Max(X)$ (otherwise interchange $y$ and $z$). Then $f\cdot_\lb (g\cdot_\mu h)=\nu(y,z)e_x\cdot_\lb e_{yz}=0$. By the same argument $(f\cdot_\mu g)\cdot_\lb h=h\cdot_\lb(f\cdot_\mu g)=0=(g\cdot_\lb h)\cdot_\mu f=f\cdot_\mu (g\cdot_\lb h)$.
	
	\textit{Case 2.} $(f,g,h)=(e_x,e_x,e_y)$ with $x\ne y$. It suffices to consider $x\in\Min(X)\sqcup\Max(X)$, otherwise all the products in \cref{(f-cdot_lb-g)-cdot_mu-h=f-cdot_lb-(g-cdot_mu-h)} are zero. Let $x\in\Min(X)$. If $x\nless y$ or $y\not\in\Max(X)$, then $e_{V_{xv}}\cdot_\mu e_y=0$ for all $(x,v)\in X^2_e$ by \cref{mu-products-with-e_V_xy}\cref{e_V_xy-cdot_mu-e_z=0}. Therefore, $(f\cdot_\lb g)\cdot_\mu h=0$. Obviously, $f\cdot_\lb (g\cdot_\mu h)=0$, as $g\cdot_\mu h=0$. Similarly, in this case $e_{xv}\cdot_\lb e_y=0$ for all $x<v$ implies $(f\cdot_\mu g)\cdot_\lb h=0$, while $f\cdot_\mu (g\cdot_\lb h)=0$ is due to $g\cdot_\lb h=0$. Let $x<y\in\Max(X)$. Then $(f\cdot_\lb g)\cdot_\mu h=-\sgn_{u_0}(x,y)\lb(x,y)e_{V_{xy}}\cdot_\mu e_y=\lb(x,y)\nu(x,y)e_{xy}$ by \cref{mu-products-with-e_V_xy}\cref{e_V_xy-cdot_mu-e_x-or-e_y} and $f\cdot_\lb (g\cdot_\mu h)=\nu(x,y)e_x\cdot_\lb e_{xy}=\lb(x,y)\nu(x,y)e_{xy}$. It is easily seen that $(f\cdot_\mu g)\cdot_\lb h=-\nu(x,y)e_{xy}\cdot_\lb e_y=\lb(x,y)\nu(x,y)e_{xy}$ and $f\cdot_\mu (g\cdot_\lb h)=\sgn_{u_0}(x,y)\lb(x,y)e_x\cdot_\mu e_{V_{xy}}=\lb(x,y)\nu(x,y)e_{xy}$ by \cref{mu-products-with-e_V_xy}\cref{e_V_xy-cdot_mu-e_x-or-e_y}. Whenever $x\in\Max(X)$, one sees that all the products in \cref{(f-cdot_lb-g)-cdot_mu-h=f-cdot_lb-(g-cdot_mu-h)} are equal to $-\lb(y,x)\nu(y,x)e_{yx}$, if $\Min(X)\ni y<x$, or $0$, otherwise.
	
	\textit{Case 3.} $(f,g,h)=(e_x,e_y,e_x)$ with $x\ne y$. By commutativity, we reduce the $4$ products \cref{(f-cdot_lb-g)-cdot_mu-h=f-cdot_lb-(g-cdot_mu-h)} to those from Case 2, where all of them are equal.
	
	\textit{Case 4.} $(f,g,h)=(e_x,e_x,e_x)$. If $x\in\Min(X)$, then by \cref{mu-products-with-e_V_xy}\cref{e_V_xy-cdot_mu-e_x-or-e_y} we have $(f\cdot_\lb g)\cdot_\mu h=-\sum_{(x,v)\in X^2_e}\sgn_{u_0}(x,v)\lb(x,v)e_{V_{xv}}\cdot_\mu e_x=-\sum_{(x,v)\in X^2_e}\lb(x,v)\nu(x,v)e_{xv}$. On the other hand, $f\cdot_\lb (g\cdot_\mu h)=-\sum_{x<v\in\Max(X)}\nu(x,v)e_x\cdot_\lb e_{xv}=-\sum_{(x,v)\in X^2_e}\lb(x,v)\nu(x,v)e_{xv}$. The equality $(f\cdot_\mu g)\cdot_\lb h=f\cdot_\mu (g\cdot_\lb h)$ follows by commutativity. If $x\in\Max(X)$, then all the products in \cref{(f-cdot_lb-g)-cdot_mu-h=f-cdot_lb-(g-cdot_mu-h)} are equal to $\sum_{(u,x)\in X^2_e}\lb(u,x)\nu(u,x)e_{ux}$.
\end{proof}

\begin{thrm}\label{descr-TP-on-T(X_K)}
	Let $\ch(K)=0$ and $X$ be connected with $|X|\ge 2$. A binary operation $\cdot$ on $I(X,K)$ is a transposed Poisson algebra structure on $I(X,K)$ if and only if $\cdot$ is the sum of a structure of Poisson type, a mutational structure and a $\lb$-structure.
\end{thrm}
\begin{proof}
	\textit{The ``if'' part.} A consequence of \cref{mutational+Poisson-TP,lb-structure-is-TP,lb+nu-structure-is-TP}.
	
	\textit{The ``only if'' part.} Let $\cdot$ be a transposed Poisson algebra structure on $I(X,K)$. By \cref{descr-half-der-I(X_K),glavlem} for all $x\le y$ in $X$ there exist $c_{xy}\in Z([I(X,K),I(X,K)])$, $\sg_{xy}:X^2_<\to K$ (constant on chains and cycles in $X$) and $\kp_{xy}:X\to K$, such that
	\begin{align*}%\label{e_ij-cdot-e_kl-general}
		e_{xy}\cdot e_{uv}=\ad_{c_{xy}}(e_{uv})+\vf_{\sg_{xy}}(e_{uv})+\vf_{\kp_{xy}}(e_{uv}).
	\end{align*}

%	Let $x<y$. Since $X$ is connected and $|X|>2$, there exist $u<v$ with $(u,v)\ne(x,y)$. Then $e_{xy}\cdot e_{uv}=\sg_{xy}(u,v)e_{uv}$ and $e_{uv}\cdot e_{xy}=\sg_{uv}(x,y)e_{xy}$. Consequently,
%	\begin{align}\label{sg_xy(u_v)=0}
%		\sg_{xy}(u,v)=0,\text{ unless }(u,v)=(x,y).
%	\end{align}
%	If $\{x,y\}$ is not a maximal chain in $X$, then there exist $u<v$, belonging to a same maximal chain with $x<y$, and such that $(u,v)\ne (x,y)$. Since $\sg_{xy}$ is constant on maximal chains in $X$, then $\sg_{xy}(x,y)=\sg_{xy}(u,v)=0$ by \cref{sg_xy(u_v)=0}. So, we have a stronger condition:
%	\begin{align}\label{sg_xy(u_v)=0-for-x_y-max-chain}
%		\sg_{xy}(u,v)=0,\text{ unless }\{x,y\}\text{ is a maximal chain and }(u,v)=(x,y).
%	\end{align}
	Let $x<y$ and $u\in X$. Then $e_u\cdot e_{xy}=\sg_{uu}(x,y)e_{xy}$, while $e_{xy}\cdot e_u=[c_{xy},e_u]+\vf_{\sg_{xy}}(e_u)+\kp_{xy}(u)\dl$, where $\vf_{\sg_{xy}}(e_u)+\kp_{xy}(u)\dl\in D(X,K)$. Hence $\vf_{\sg_{xy}}(e_u)=-\kp_{xy}(u)\dl\in Z(I(X,K))$, which due to \cref{vf_sg(e_x)(u_0_u_0)=0} yields
	\begin{align*}%\label{vf_sg_xy_kp_xy|_D(X_K)=0}
		\vf_{\sg_{xy}}|_{D(X,K)}=0\text{ and }\kp_{xy}=0\text{ for }x<y.
	\end{align*}
It follows from \cref{vf(D)-sst-Z<=>vf([I_I])=0} that $\sg_{xy}(u,v)e_{uv}=\vf_{\sg_{xy}}(e_{uv})=0$ for all $u<v$, that is,
\begin{align*}%\label{sg_xy=0-for-x<y}
	\sg_{xy}=0\text{ for }x<y.
\end{align*}
As a consequence,
\begin{align}\label{e_xy-cdot-e_uv=0}
	e_{xy}\cdot e_{uv}=0\text{ for }x<y\text{ and }u<v.
\end{align}
	Assuming, moreover, that $u\not\in\Min(X)\sqcup\Max(X)$, we have $[c_{xy},e_u]=0$, so
	\begin{align}\label{sg_uu=0-u-not-min-or-max}
		\sg_{uu}=0\text{ for }u\not\in\Min(X)\sqcup\Max(X).
	\end{align}
	Now take $x<y$ and $u\in\Min(X)$. Then $e_u\cdot e_{xy}=\sg_{uu}(x,y)e_{xy}$, while $e_{xy}\cdot e_u=-\sum_{u<v}c_{xy}(u,v)e_{uv}$. Similarly, for any $v\in\Max(X)$ we have $e_v\cdot e_{xy}=\sg_{vv}(x,y)e_{xy}$, while $e_{xy}\cdot e_v=\sum_{u<v}c_{xy}(u,v)e_{uv}$. Hence, for all $x<y$,
	\begin{align}
%		\sg_{uu}(x,y)&=c_{xy}=0\text{ for }u\in\Min(X)\sqcup\Max(X)\text{ and }x<y,\ x\not\in\Min(X)\text{ or }y\not\in\Max(X),\\
%		\sg_{uu}(x,y)&=c_{xy}(u,v)=0\text{ for }u\in\Min(X),\ u\ne x\in\Min(X)\text{ and }y\in\Max(X),\\
%		\sg_{xx}(x,y)&=-c_{xy}(x,y),\ c_{xy}(x,v)=0\text{ for }x\in\Min(X)\text{ and }v\ne y\in\Max(X),\\
%		\sg_{vv}(x,y)&=c_{xy}(u,v)=0\text{ for }v\in\Max(X),\ x\in\Min(X)\text{ and }v\ne y\in\Max(X),\\
%		\sg_{yy}(x,y)&=c_{xy}(x,y),\ c_{xy}(u,y)=0\text{ for }u\ne x\in\Min(X)\text{ and }y\in\Max(X).\\
		\sg_{uu}(x,y)&=0,\text{ unless }x\in\Min(X),\ y\in\Max(X)\text{ and }u\in\{x,y\},\label{sg_uu(x_y)=0}\\
		c_{xy}(u,v)&=0,\text{ unless }x\in\Min(X),\ y\in\Max(X)\text{ and }(u,v)=(x,y),\label{c_xy(u_v)=0}\\
		\sg_{xx}(x,y)&=-c_{xy}(x,y)=-\sg_{yy}(x,y),\text{ if }x\in\Min(X)\text{ and }y\in\Max(X).\notag%\label{sg_xx(x_y)=-c_xy(x_y)=-sg_yy(x_y)}
	\end{align}
If there exists $x<z<y$, then $\sg_{xx}(x,y)=\sg_{xx}(x,z)=0$ by \cref{sg_uu(x_y)=0} and the fact that $\sg_{xx}$ is constant on chains in $X$. Moreover, if $l(x,y)=1$ and there exists a cycle containing $x<y$, then by \cref{cycle-containing-x-and-y} there is a cycle $u_0,\dots,u_m=u_0$ having $(x,y)$ as an edge, so choosing $0\le i<m$ with $x\not\in\{u_i,u_{i+1}\}$ (this is possible, because a cycle has at least $4$ distinct elements) and applying \cref{sg_uu(x_y)=0} we again get $\sg_{xx}(x,y)=\sg_{xx}(u_i,u_{i+1})=0$, if $u_i<u_{i+1}$, or $\sg_{xx}(x,y)=\sg_{xx}(u_{i+1},u_i)=0$, if $u_i>u_{i+1}$, because $\sg_{xx}$ is constant on cycles in $X$. Therefore, the condition ``$x\in\Min(X)$ and $y\in\Max(X)$'' in \cref{sg_uu(x_y)=0,c_xy(u_v)=0} can be replaced by a stronger one: ``$(x,y)\in X^2_e$''. As a consequence, denoting 
\begin{align*}%\label{lb_xy=sg_xx(x_y)}
	\lb(x,y):=\sg_{xx}(x,y),\ \text{if }(x,y)\in X^2_e,
\end{align*}
we have
\begin{align}
	e_{xy}\cdot e_u&=0\text{ for }x<y,\text{ unless }(x,y)\in X^2_e\text{ and }u\in\{x,y\},\label{e_xy-cdot-e_u=0}\\
	e_x\cdot e_{xy}&=-e_{xy}\cdot e_y=\lb(x,y)e_{xy},\text{ if }(x,y)\in X^2_e.\label{e_xy-cdot-e_y=lb_xy.e_xy}
\end{align}

Now let us study $e_x\cdot e_y$. To this end, given $x,y\in X$, introduce
\begin{align*}%\label{mu_xy=vf_(sg_xx_kp_xx)(e_y)(x_1_x_1)}
	\mu(x,y):=\kp_{xx}(y),%\ x\le y,\ x\not\in \Min(X)\sqcup\Max(X)\text{ or }y\not\in \Min(X)\sqcup\Max(X),
\end{align*}
Let $x\not\in\Min(X)\sqcup\Max(X)$. In view of \cref{sg_uu=0-u-not-min-or-max} we have $\vf_{\sg_{xx}}(e_{uv})=0$ for all $u<v$, so that $\vf_{\sg_{xx}}(e_y)\in Z(I(X,K))$ for all $y\in X$ by \cref{vf(D)-sst-Z<=>vf([I_I])=0}. Thanks to \cref{vf_sg(e_x)(u_0_u_0)=0}, we conclude that $\vf_{\sg_{xx}}(e_y)=0$, so
\begin{align}\label{e_x-cdot-e_y=mu_xy.dl-x-not-min-or-max}
	e_x\cdot e_y=\vf_{\kp_{xx}}(e_y)=\mu(x,y)\dl,\text{ if }x\not\in \Min(X)\sqcup\Max(X)\text{ or }y\not\in \Min(X)\sqcup\Max(X).
\end{align}

Let $x\in\Min(X)$ and $y\in\Max(X)$. It follows from $e_x\cdot e_y=e_y\cdot e_x$ that $\sum_{u<y}c_{xx}(u,y)e_{uy}=[c_{xx},e_y]=[c_{yy},e_x]=-\sum_{x<v}c_{yy}(x,v)e_{xv}$ and $(\vf_{\sg_{xx}}+\vf_{\kp_{xx}})(e_y)=(\vf_{\sg_{yy}}+\vf_{\kp_{yy}})(e_x)$. Hence,
\begin{align}
	c_{xx}(u,v)&=0\text{ for }u\ne x,\ c_{yy}(u,v)=0\text{ for }v\ne y,\label{c_xx(u_v)=c_yy(u_v)=0}\\
	c_{xx}(x,y)&=-c_{yy}(x,y),\text{ if }x<y.\label{c_xx(x_y)=-c_yy(x_y)}
\end{align}
Denoting 
\begin{align*}
	\nu:=\sum_{x<y}c_{xx}(x,y)e_{xy},
\end{align*}
by \cref{c_xx(u_v)=c_yy(u_v)=0,c_xx(x_y)=-c_yy(x_y)} we have for all $x,y\in X$:
\begin{align}\label{(e_x-cdot-e_y)_J=nu-structure}
	(e_x\cdot e_y)_J=[c_{xx},e_y]=[[e_x,\nu],e_y].
%	\begin{cases}
%		-\nu(x,y)e_{xy}, & \Min(X)\ni x<y\in\Max(X),\\
%		\sum_{x<v\in\Max(X)}\nu(x,v)e_{xv}, & x=y\in\Min(X),\\
%		\sum_{\Min(X)\ni u<x}\nu(u,x)e_{ux}, & x=y\in\Max(X),\\
%		0, & \mbox{otherwise}.
%	\end{cases}
%	(e_x\cdot e_y)_J&=[c_{xx},e_y]=c_{xx}(x,y)=-\nu(x,y)e_{xy},\mbox{ if }x\in\Min(X)\mbox{ and }y\in\Max(X),\\
%	(e_x\cdot e_x)_J&=[c_{xx},e_x]=-\sum_{x<v\in\Max(X)}c_{xx}(x,v)e_{xv}=\sum_{x<v\in\Max(X)}\nu(x,v)e_{xv},\mbox{ if }x\in\Min(X),\\
%	(e_x\cdot e_y)_J&=[c_{xx},e_y]=-\sum_{y<v\in\Max(X)}c_{xx}(y,v)e_{yv}=0,\mbox{ if }x,y\in\Min(X),\ x\ne y,\\
%	(e_y\cdot e_y)_J&=[c_{yy},e_y]=\sum_{\Min(X)\ni u<y}c_{yy}(u,y)=\sum_{\Min(X)\ni u<y}\nu_{uy}e_{uy},\mbox{ if }y\in\Max(X).
\end{align}

It remains to determine $(e_x\cdot e_y)_D=\vf_{\sg_{xx}}(e_y)+\vf_{\kp_{xx}}(e_y)$. Observe that $\vf_{\kp_{xx}}(e_y)=\mu(x,y)\dl$, while $\vf_{\sg_{xx}}(e_y)$ can be calculated using \cref{e_x-cdot-e_y=sums-with-lambdas}, because $\sg_{xx}$ has exactly the same form as $\sg_x$ defined in \cref{sg_x(u_v)=lb_xv-or-lb_ux}. Thus,
\begin{align}\label{(e_x-cdot-e_y)_D=lb-structure}
	(e_x\cdot e_y)_D=(e_y\cdot e_x)_D=
	\begin{cases}
		\mu(x,y)\dl+\sgn_{u_0}(x,y)\lb(x,y)e_{V_{xy}},& (x,y)\text{ is extreme},\\
		\mu(x,y)\dl-\sum_{x<v}\sgn_{u_0}(x,v)\lb(x,v)e_{V_{xv}}, & x=y\in\Min(X),\\
		\mu(x,y)\dl-\sum_{u<x}\sgn_{u_0}(u,x)\lb(u,x)e_{V_{ux}}, & x=y\in\Max(X),\\
		\mu(x,y)\dl, & \text{otherwise}.
	\end{cases}
\end{align}
Combining \cref{e_xy-cdot-e_uv=0,e_xy-cdot-e_u=0,e_xy-cdot-e_y=lb_xy.e_xy,e_x-cdot-e_y=mu_xy.dl-x-not-min-or-max,(e_x-cdot-e_y)_J=nu-structure,(e_x-cdot-e_y)_D=lb-structure}, we get the desired form of the structure $\cdot$.
\end{proof}

\begin{rem}
	The automorphism $\phi\in\Aut(I(X,K),[\cdot,\cdot])$ from \cref{nu_xy-0-or-1} does not affect the Poisson and $\lb$-structures from the decomposition of $\cdot$, so one can assume that $\nu\in Z([I(X,K),I(X,K)])$ that defines the mutational structure in \cref{descr-TP-on-T(X_K)} satisfies $\nu(x,y)\in\{0,1\}$ for all $\Min(X)\ni x<y\in\Max(X)$.
\end{rem}

\begin{exm}
	Let $X=\{1,\dots,n\}$ be a chain. 
	
	If $n>2$, then $X^2_e=\emptyset$, so any transposed Poisson structure on $(I(X,K),[\cdot,\cdot])$ is (isomorphic to) the sum of a structure of Poisson type $e_i\cdot e_j=\mu(i,j)\dl$ and a mutational structure $e_i\cdot e_j=[[e_i,\nu],e_j]$ with $\nu=\nu(1,n)e_{1n}$, namely: $e_1\cdot e_1=-e_1\cdot e_n=e_n\cdot e_n=-\nu(1,n)e_{1n}$, where $\nu(1,n)\in\{0,1\}$. Thus, up to an isomorphism, we recover \cite[Theorem 11]{KK7}.
	
	If $n=2$, then $X^2_e=\{(1,2)\}$. Choosing $u_0=1$, we have $V_{12}=\{2\}$ and $\sgn_1(1,2)=1$. Apart from the two above structures, we also have a $\lb$-structure  $e_1\cdot e_{12}=-e_{12}\cdot e_2=\lb(1,2)e_{12}$, $e_1\cdot e_1=-e_1\cdot e_2=e_2\cdot e_2=-\lb(1,2)e_2$. Observe that $e_1\cdot e_1=\lb(1,2)e_1-\lb(1,2)\dl$, so we can slightly modify the Poisson part of $\cdot$ (replacing $\mu(1,1)$ by $\mu(1,1)-\lb(1,2)$) to get exactly the family from \cite[Theorem 12 (iii)]{KK7} with $c=\lb(1,2)$.
\end{exm}

\begin{exm}
	Let $X=\{1,2,3\}$ with the following Hasse diagram.
	
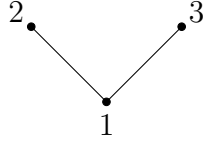
\begin{center}
		\begin{tikzpicture}
			\draw  (0,0)-- (-1,1);
			\draw  (0,0)-- (1,1);
			\draw [fill=black] (-1,1) circle (0.05);
			\draw  (-1.2,1.2) node {$2$};
			\draw [fill=black] (1,1) circle (0.05);
			\draw  (1.2,1.2) node {$3$};
			\draw [fill=black] (0,0) circle (0.05);
			\draw  (0,-0.3) node {$1$};
		\end{tikzpicture}
	\end{center}
Then $X^2_e=\{(1,2),(1,3)\}$. Choosing $u_0=1$, we have $V_{12}=\{2\}$, $V_{13}=\{3\}$ and $\sgn_1(1,2)=\sgn_1(1,3)=1$. Thus, any transposed Poisson structure on $I(X,K)$ is the sum of a structure of Poisson type $e_i\cdot e_j=\mu(i,j)\dl$, a mutational structure $e_i\cdot e_j=[[e_i,\nu],e_j]$ with $\nu=\nu(1,2)e_{12}+\nu(1,3)e_{13}$, namely
\begin{center}
	\begin{tabular}{rclrclrcl}
		$e_1\cdot e_1$ & $=$ & $-\nu(1,2)e_{12} - \nu(1,3)e_{13}$, &
		$e_1\cdot e_2$ & $=$ & $\nu(1,2)e_{12}$, &
		$e_1\cdot e_3$ & $=$ & $\nu(1,3)e_{13}$,\\ 
		$e_2\cdot e_2$ & $=$ & $-\nu(1,2)e_{12}$, &
		$e_3\cdot e_3$ & $=$ & $-\nu(1,3)e_{13}$,
	\end{tabular}
\end{center}
%\begin{align*}
%	e_1\cdot e_1&=-\nu(1,2)e_{12} - \nu(1,3)e_{13}, 
%	e_1\cdot e_2 = \nu(1,2)e_{12}, 
%	e_1\cdot e_3 = \nu(1,3)e_{13},\\ 
%	e_2\cdot e_2 &= -\nu(1,2)e_{12}, 
%	e_3\cdot e_3 = -\nu(1,3)e_{13}
%\end{align*} 
where $\nu(1,2),\nu(1,3)\in\{0,1\}$, and a $\lb$-structure determined by $\lb:X^2_e\to K$ as follows:
\begin{center}
	\begin{tabular}{rclrclrcl}
		$e_1\cdot e_{12}$ & $=$ & $-e_{12}\cdot e_2	=\lb(1,2)e_{12}$, &
		$e_1\cdot e_{13}$ & $=$ & $-e_{13}\cdot e_3=\lb(1,3)e_{13}$,& \\
		$e_1\cdot e_1$ & $=$ & $-\lb(1,2)e_2 - \lb(1,3)e_3$, & 
		$e_1\cdot e_2$ & $=$ & $\lb(1,2)e_2$, & 
		$e_1\cdot e_3$ & $=$ & $\lb(1,3)e_3$, \\ 
		$e_2\cdot e_2$ & $=$ & $-\lb(1,2)e_2$, & 
		$e_3\cdot e_3$ & $=$ & $-\lb(1,3)e_3$.
	\end{tabular}
\end{center}

%\begin{align*}
%	e_1\cdot e_{12}&=-e_{12}\cdot e_2=\lb(1,2)e_{12}, 
%	e_1\cdot e_{13}=-e_{13}\cdot e_3=\lb(1,3)e_{13},\\
%	e_1\cdot e_1&=-\lb(1,2)e_2 - \lb(1,3)e_3, 
%	e_1\cdot e_2 = \lb(1,2)e_2, 
%	e_1\cdot e_3 = \lb(1,3)e_3, 
%	e_2\cdot e_2 = -\lb(1,2)e_2, 
%	e_3\cdot e_3 = -\lb(1,3)e_3.
%\end{align*}
\end{exm}

\begin{exm}
	Let $X=\{1,2,3,4\}$ with the following Hasse diagram.
	
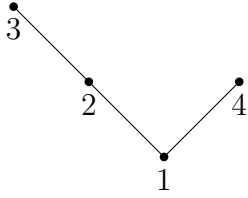
\begin{center}
		\begin{tikzpicture}
			\draw  (0,0)-- (-1,1);
			\draw  (0,0)-- (1,1);
			\draw  (-1,1)--(-2,2);
			\draw [fill=black] (-1,1) circle (0.05);
			\draw  (-1,0.7) node {$2$};
			\draw [fill=black] (1,1) circle (0.05);
			\draw  (1,0.7) node {$4$};
			\draw [fill=black] (0,0) circle (0.05);
			\draw  (0,-0.3) node {$1$};
			\draw [fill=black] (-2,2) circle (0.05);
			\draw  (-2,1.7) node {$3$};
		\end{tikzpicture}
	\end{center}
	Then $X^2_e=\{(1,4)\}$. Choosing $u_0=1$, we have $V_{14}=\{4\}$ and $\sgn_1(1,4)=1$. Thus, any transposed Poisson structure on $I(X,K)$ is the sum of a structure of Poisson type $e_i\cdot e_j=\mu(i,j)\dl$, a mutational structure $e_i\cdot e_j=[[e_i,\nu],e_j]$ with $\nu=\nu(1,3)e_{13}+\nu(1,4)e_{14}$, namely
	\begin{center}
		\begin{tabular}{rclrclrcl}
			$e_1\cdot e_1$ & $=$ & $-\nu(1,3)e_{13} - \nu(1,4)e_{14}$, & 
			$e_1\cdot e_3$ & $=$ & $\nu(1,3)e_{13}$, & 
			$e_1\cdot e_4$ & $=$ & $\nu(1,4)e_{14}$,\\ 
			$e_3\cdot e_3$ & $=$ & $-\nu(1,3)e_{13}$, & 
			$e_4\cdot e_4$ & $=$ & $-\nu(1,4)e_{14}$,
		\end{tabular}
	\end{center}
%	\begin{align*}
%		e_1\cdot e_1=-\nu(1,3)e_{13} - \nu(1,4)e_{14}, 
%		e_1\cdot e_3 = \nu(1,3)e_{13}, 
%		e_1\cdot e_4 = \nu(1,4)e_{14}, 
%		e_3\cdot e_3 = -\nu(1,3)e_{13}, 
%		e_4\cdot e_4 = -\nu(1,4)e_{14}
%	\end{align*} 
	where $\nu(1,3),\nu(1,4)\in\{0,1\}$, and a $\lb$-structure determined by $\lb:X^2_e\to K$ as follows:
	\begin{center}
		\begin{tabular}{rclrclrcl}
			$e_1\cdot e_{14}$ & $=$ & $-e_{14}\cdot e_4=\lb(1,4)e_{14}$, & &&&&& \\
			$e_1\cdot e_1$ & $=$ & $-\lb(1,4)e_4$, & 
			$e_1\cdot e_4$ & $=$ & $\lb(1,4)e_4$, &
			$e_4\cdot e_4$ & $=$ & $-\lb(1,4)e_4$.
		\end{tabular}
	\end{center}
%	\begin{align*}
%		e_1\cdot e_{14}=-e_{14}\cdot e_4=\lb(1,4)e_{14}, 
%		e_1\cdot e_1 = -\lb(1,4)e_4, 
%		e_1\cdot e_4 = \lb(1,4)e_4, 
%		e_4\cdot e_4 = -\lb(1,4)e_4.
%	\end{align*}
\end{exm}

\begin{exm}
	Let $X=\{1,2,3,4\}$ with the following Hasse diagram.
	
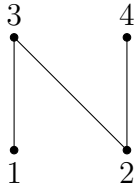
\begin{center}
		\begin{tikzpicture}
			\draw  (0,0)-- (0,1.5);
			\draw  (0,1.5)-- (1.5,0);
			\draw  (1.5,0)-- (1.5,1.5);
			\draw [fill=black] (0,0) circle (0.05);
			\draw  (0,-0.3) node {$1$};
			\draw [fill=black] (1.5,0) circle (0.05);
			\draw  (1.5,-0.3) node {$2$};
			\draw [fill=black] (0,1.5) circle (0.05);
			\draw  (0,1.8) node {$3$};
			\draw [fill=black] (1.5,1.5) circle (0.05);
			\draw  (1.5,1.8) node {$4$};
		\end{tikzpicture}
	\end{center}
	Then $X^2_e=\{(1,3),(2,3),(2,4)\}$. Choosing $u_0=1$, we have $V_{13}=\{2,3,4\}$, $V_{23}=\{2,4\}$, $V_{24}=\{4\}$ and $\sgn_1(1,3)=-\sgn_1(2,3)=\sgn_1(2,4)=1$. Thus, any transposed Poisson structure on $I(X,K)$ is the sum of a structure of Poisson type $e_i\cdot e_j=\mu(i,j)\dl$, a mutational structure $e_i\cdot e_j=[[e_i,\nu],e_j]$ with $\nu=\nu(1,3)e_{13}+\nu(2,3)e_{23}+\nu(2,4)e_{24}$, namely 
	\begin{center}
		\begin{tabular}{rclrclrcl}
			$e_1\cdot e_1$ & $=$ & $-\nu(1,3)e_{13}$, & 
			$e_1\cdot e_3$ & $=$ & $\nu(1,3)e_{13}$, & 
			$e_2\cdot e_2$ & $=$ & $-\nu(2,3)e_{23} - \nu(2,4)e_{24}$,\\ 
			$e_2\cdot e_3$ & $=$ & $\nu(2,3)e_{23}$, & 
			$e_2\cdot e_4$ & $=$ & $\nu(2,4)e_{24}$, & 
			$e_3\cdot e_3$ & $=$ & $-\nu(1,3)e_{13} - \nu(2,3)e_{23}$,\\ 
			$e_4\cdot e_4$ & $=$ & $-\nu(2,4)e_{24}$,
		\end{tabular}
	\end{center}
%	\begin{align*}
%		e_1\cdot e_1&=-\nu(1,3)e_{13}, 
%		e_1\cdot e_3 = \nu(1,3)e_{13}, 
%		e_2\cdot e_2 = -\nu(2,3)e_{23} - \nu(2,4)e_{24}, 
%		e_2\cdot e_3 =\nu(2,3)e_{23},\\ 
%		e_2\cdot e_4 &=\nu(2,4)e_{24}, 
%		e_3\cdot e_3 = -\nu(1,3)e_{13} - \nu(2,3)e_{23}, 
%		e_4\cdot e_4 = -\nu(2,4)e_{24}
%	\end{align*}
	where $\nu(1,3),\nu(2,3),\nu(2,4)\in\{0,1\}$, and a $\lb$-structure determined by $\lb:X^2_e\to K$ as follows:
	\begin{center}
		\begin{tabular}{rclrcl}
			$e_1\cdot e_{13}$ & $=$ & $-e_{13}\cdot e_3=\lb(1,3)e_{13}$, & 
			$e_2\cdot e_{23}$ & $=$ & $-e_{23}\cdot e_3=\lb(2,3)e_{23}$, \\ 
			$e_2\cdot e_{24}$ & $=$ & $-e_{24}\cdot e_4=\lb(2,4)e_{24}$,&
			$e_1\cdot e_1$ & $=$ & $-\lb(1,3)(e_2+e_3+e_4)$, \\ 
			$e_1\cdot e_3$ & $=$ & $\lb(1,3)(e_2+e_3+e_4)$,& 
			$e_2\cdot e_2$ & $=$ & $\lb(2,3)(e_2+e_4)-\lb(2,4)e_4$, \\  %$\lb(2,3)e_2+(\lb(2,3)-\lb(2,4))e_4$
			$e_2\cdot e_3$ & $=$ & $-\lb(2,3)(e_2+e_4)$, &
			$e_2\cdot e_4$ & $=$ & $\lb(2,4)e_4$, \\ 
			$e_3\cdot e_3$ & $=$ & $-\lb(1,3)(e_2+e_3+e_4)+\lb(2,3)(e_2+e_4)$, & %$(\lb(2,3)-\lb(1,3))e_2-\lb(1,3)e_3+(\lb(2,3)-\lb(1,3))e_4$
			$e_4\cdot e_4$ & $=$ & $-\lb(2,4)e_4$. 
		\end{tabular}
	\end{center}
%	\begin{align*}
%		e_1\cdot e_{13}&=-e_{13}\cdot e_3=\lb(1,3)e_{13}, 
%		e_2\cdot e_{23}=-e_{23}\cdot e_3=\lb(2,3)e_{23}, 
%		e_2\cdot e_{24}=-e_{24}\cdot e_4=\lb(2,4)e_{24},\\
%		e_1\cdot e_1 &= -\lb(1,3)(e_2+e_3+e_4), 
%		e_1\cdot e_3 = \lb(1,3)(e_2+e_3+e_4), 
%		e_2\cdot e_2 =\lb(2,3)e_2+(\lb(2,3)-\lb(2,4))e_4,\\  %\lb(2,3)(e_2+e_4)-\lb(2,4)e_4
%		e_2\cdot e_3 &= -\lb(2,3)(e_2+e_4), 
%		e_2\cdot e_4 = \lb(2,4)e_4, 
%		e_3\cdot e_3 =(\lb(2,3)-\lb(1,3))e_2-\lb(1,3)e_3+(\lb(2,3)-\lb(1,3))e_4,\\ %-\lb(1,3)(e_2+e_3+e_4)+\lb(2,3)(e_2+e_4)
%		e_4\cdot e_4 &= -\lb(2,4)e_4.
%	\end{align*}
\end{exm}

\begin{exm}
	Let $X=\{1,2,3,4\}$ with the following Hasse diagram.
	
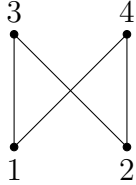
\begin{center}
		\begin{tikzpicture}
			\draw  (0,0)-- (0,1.5);
			\draw  (0,1.5)-- (1.5,0);
			\draw  (1.5,0)-- (1.5,1.5);
			\draw  (1.5,1.5)-- (0,0);
			\draw [fill=black] (0,0) circle (0.05);
			\draw  (0,-0.3) node {$1$};
			\draw [fill=black] (1.5,0) circle (0.05);
			\draw  (1.5,-0.3) node {$2$};
			\draw [fill=black] (0,1.5) circle (0.05);
			\draw  (0,1.8) node {$3$};
			\draw [fill=black] (1.5,1.5) circle (0.05);
			\draw  (1.5,1.8) node {$4$};
		\end{tikzpicture}
	\end{center}
	Then $X^2_e=\emptyset$, because the elements of $X$ form a cycle. Thus, any transposed Poisson structure on $I(X,K)$ is the sum of a structure of Poisson type $e_i\cdot e_j=\mu(i,j)\dl$ and a mutational structure $e_i\cdot e_j=[[e_i,\nu],e_j]$ with $\nu=\nu(1,3)e_{13}+\nu(1,4)e_{14}+\nu(2,3)e_{23}+\nu(2,4)e_{24}$, namely
	\begin{center}
		\begin{tabular}{rclrclrcl}
			$e_1\cdot e_1$ & $=$ & $-\nu(1,3)e_{13}-\nu(1,4)e_{14}$, & 
			$e_1\cdot e_3$ & $=$ & $\nu(1,3)e_{13}$, & 
			$e_1\cdot e_4$ & $=$ & $\nu(1,4)e_{14}$,\\
			$e_2\cdot e_2$ & $=$ & $-\nu(2,3)e_{23} - \nu(2,4)e_{24}$, &
			$e_2\cdot e_3$ & $=$ & $\nu(2,3)e_{23}$, & 
			$e_2\cdot e_4$ & $=$ & $\nu(2,4)e_{24}$,\\ 
			$e_3\cdot e_3$ & $=$ & $-\nu(1,3)e_{13} - \nu(2,3)e_{23}$, & 
			$e_4\cdot e_4$ & $=$ & $-\nu(1,4)e_{14}-\nu(2,4)e_{24}$,
		\end{tabular}
	\end{center}
%	\begin{align*}
%		e_1\cdot e_1 &= -\nu(1,3)e_{13}-\nu(1,4)e_{14}, 
%		e_1\cdot e_3 = \nu(1,3)e_{13}, 
%		e_1\cdot e_4 = \nu(1,4)e_{14},
%		e_2\cdot e_2 = -\nu(2,3)e_{23} - \nu(2,4)e_{24}, \\
%		e_2\cdot e_3 &= \nu(2,3)e_{23}, 
%		e_2\cdot e_4 = \nu(2,4)e_{24}, 
%		e_3\cdot e_3 = -\nu(1,3)e_{13} - \nu(2,3)e_{23}, 
%		e_4\cdot e_4 = -\nu(1,4)e_{14}-\nu(2,4)e_{24}
%	\end{align*}
	where $\nu(1,3),\nu(1,4),\nu(2,3),\nu(2,4)\in\{0,1\}$.
\end{exm}
	%	\medskip 

	%5	{\bf Compliance with ethical standard}
	
	%	\medskip 
	
	%	{\bf Author contributions} 	All authors contributed to the study, conception and	design. All authors read and approved the final manuscript.

	%	{\bf Conflict of interest} 	There is no potential conflict of ethical approval, conflict	of interest, and ethical standards.

	%	{\bf Data Availibility} 	Data sharing is not applicable to this article as no datasets	were generated or analyzed during the current study.

	%	%\newpage

\end{document}